\documentclass[a4paper,12pt]{elsarticle}

\usepackage{amsfonts}
\usepackage{amsthm,amsfonts,epsfig,setspace}
\usepackage{colortbl,color}
\usepackage[margin=1in]{geometry}
\usepackage{epstopdf}
\usepackage{amssymb}
\usepackage{mathrsfs}
\usepackage{amsfonts,latexsym,bm}
\usepackage{amsmath}
\usepackage{amsthm}
\usepackage{float}
\usepackage{graphicx}
\usepackage{CJK}
\usepackage{graphics}
\usepackage{color}
\usepackage[english]{babel}

\usepackage{algorithm}
\floatname{algorithm}{\normalfont\sffamily\bfseries Algorithm}
\usepackage{algpseudocode}

\usepackage{subcaption}

\makeatletter 
\@addtoreset{equation}{section}
\makeatother  

\newtheorem{proposition}{Proposition}

\newcommand{\replyone}[1]{{\color{black}#1}}
\newcommand{\replytwo}[1]{{\color{black}#1}}

\begin{document}

\title{A multilayer level-set method for eikonal-based traveltime tomography}

\author[1]{Wenbin Li} \ead{liwenbin@hit.edu.cn}
\author[2]{Ken K.T. Hung} \ead{ktkhung@connect.ust.hk}
\author[2]{Shingyu Leung\corref{cor1}} \ead{masyleung@ust.hk}

\address[1]{School of Science, Harbin Institute of Technology, Shenzhen, Shenzhen, 518055, China}
\address[2]{Department of Mathematics, Hong Kong University of Science and Technology, Clear Water Bay, Hong Kong}

\cortext[cor1]{Corresponding author}

\begin{abstract}
We present a novel multilayer level-set method (MLSM) for eikonal-based first-arrival traveltime tomography.
Unlike classical level-set approaches that rely solely on the zero-level set, the MLSM represents multiple phases through a sequence of $i_n$-level sets ($n = 0, 1, 2, \cdots$).
Near each $i_n$-level set, the function is designed to behave like a local signed-distance function, enabling a single level-set formulation to capture arbitrarily many interfaces and subregions.
Within this Eulerian framework, first-arrival traveltimes are computed as viscosity solutions of the eikonal equation, and Fr\'{e}chet derivatives of the misfit are obtained via the adjoint state method.
To stabilize the inversion, we incorporate several regularization strategies, including multilayer reinitialization, arc-length penalization, and Sobolev smoothing of model parameters.
In addition, we introduce an illumination-based error measure to assess reconstruction quality.
Numerical experiments demonstrate that the proposed MLSM efficiently recovers complex discontinuous slowness models with multiple phases and interfaces.
\end{abstract}

\thispagestyle{plain} \maketitle 

\section{Introduction}
\label{Sec:Introduction}
Traveltime tomography is a widely used method in geophysical and medical imaging that determines the spatial distribution of a medium's wave-propagation velocity (or its reciprocal, slowness) by inverting measured wave traveltimes~\cite{borgerlinetc87,zelsmi92,leuqia07,glosha77,hoocar14}.
This technique is broadly applicable to different wave phenomena, including seismic, electromagnetic, and acoustic waves.
The method is predicated on the principle that the traveltime between a source and a receiver represents a path integral of the medium's slowness, such that variations in the measured traveltimes reveal heterogeneity within the medium being imaged.
Due to the multiple ray paths that can connect a source and \replyone{a} receiver, traveltime tomography can be formulated using either first-arrival~\cite{leuqia06,tainobchacal09} or multi-arrival data~\cite{dellai95,leuqia07}.
For this work, we focus specifically on the inverse problem of first-arrival traveltime tomography.

Traditional traveltime tomography methods primarily rely on the Lagrangian ray-tracing technique, solving a set of ordinary differential equations derived from Fermat's principle to calculate the trajectory of wave rays and their traveltime~\cite{nol09}.
The technique is termed Lagrangian because it explicitly tracks the motion of rays through space, in contrast to Eulerian methods which observe properties at fixed spatial points. 
A major drawback of the ray-tracing approach is its non-uniform resolution, where ray paths are highly sensitive to initial conditions and the medium's velocity structure.
In \replyone{a} media with strong heterogeneity, the ray-tracing technique inevitably produces an irregular distribution of ray paths.
This uneven coverage leaves portions of the model poorly constrained (e.g., forming shadow zones), and may result in an ill-conditioned system of equations in the traveltime tomography~\cite{ber00a,ber00b}. 

An alternative approach is to compute the traveltime in an Eulerian framework.
This is achieved by solving the eikonal equation across a computational grid, where typical approaches include the fast marching method~\cite{set99siamrev} and the fast sweeping method~\cite{zha05, qiazhazha07, zha07, detmilgibmin13, liqia20newton}.
Based on the fast sweeping method,~\cite{leuqia06} \replyone{has developed} a \replyone{Partial Differential Equation (PDE)}-based Eulerian approach to traveltime tomography using first arrivals. 
In that work, the first-arrival traveltime \replyone{was} modeled as the viscosity solution~\cite{cralio83} of the eikonal equation, which \replyone{was} computed efficiently by the fast sweeping method.
The traveltime tomography \replyone{was} then performed by minimizing the data misfit under the constraint of \replyone{the} eikonal equation, where an adjoint state method~\cite{ple06} \replyone{was} proposed to solve the PDE-constrained optimization problem.
To stabilize the gradient flow, a Sobolev regularization \replyone{was} imposed on the velocity perturbations, and so the algorithm of~\cite{leuqia06} \replyone{led} to a smooth velocity model.

In practice, many important velocity/slowness models are characterized by sharp discontinuities representing faults, layer boundaries, or interfaces between distinct units.
For discontinuous velocity/slowness models,~\cite{ost00} has introduced the concept of an extended viscosity solution to justify the well-posedness of the eikonal equation.
This solution exists and is Lipschitz continuous when the slowness function is piecewise Lipschitz continuous, and it corresponds to the first-arrival traveltime.
Moreover, any finite difference scheme that is monotonic in slowness and converges to the viscosity solution for continuous slowness must also converge to the extended viscosity solution for discontinuous slowness~\cite{kaiost01}. 
In~\cite{lileu13}, we have developed a level-set adjoint state method for first-arrival traveltime tomography where the underlying slowness is discontinuous.
In that work, the level-set method \replyone{was} employed to represent a piecewise continuous slowness function.
The extended viscosity solution of the eikonal equation \replyone{was} computed to match the first-arrival traveltime, and an adjoint state method similar to that of~\cite{leuqia06} \replyone{was} proposed to solve the tomography problem.
In~\cite{lileuqia14}, we \replyone{extended} the level-set adjoint state method to solve crosswell transmission-reflection traveltime tomography problems, where the level-set function describes \replyone{both the} slowness discontinuities and the geometry of reflectors.

The level-set method~\cite{oshset88,oshfed06} has been widely used in various fields involving topological changes and interface structures, including multiphase flow simulation, shape optimization, computational geometry, and computer graphics.
In terms of inverse problems, the level-set method \replyone{was} first used for inverse obstacle problems in~\cite{san96}, and thereafter, it \replyone{was} applied to a variety of inverse problems.
For instance, in~\cite{litlessan98}, it \replyone{was} used to reconstruct 2-D binary obstacles; in~\cite{housolzha04} and~\cite{dorles06}, it \replyone{was} used for inverse scattering to determine the geometry of extended targets; in~\cite{benmil07} for electrical resistance tomography in medical imaging; in~\cite{isaleuqia11,isaleuqia13,liqia21} for inverse gravimetry; and in~\cite{liqiali20,liwanfan22} for magnetic inverse problems.

Classical level-set methods represent interfaces via the zero-level set of a single level-set function, where the interfaces (connected or \replyone{disjoint}) partition the domain into two phases. 
However, many practical problems involve heterogeneous structures with multiple phases and interacting interfaces.
For instance, the Dirichlet $k$-partition problem seeks to partition a domain into $k$ distinct subdomains to minimize the sum of their smallest eigenvalues~\cite{chuleu21}.
In tomography problems, there \replyone{are} significant interest\replyone{s} in recovering \replyone{a} media with multi-component and multilayered architectures.
To represent such \replyone{multiphase} structures, the multiple level-set method has been developed~\cite{zhachamerosh96,vescha02, decleitai09,vanasclei10}.
This approach introduces more than one level-set functions to define the interfaces between distinct components.
For instance, in~\cite{zhachamerosh96}, it associates to each \replyone{subregion} a level-set function, requiring $n$ level-set functions for $n$ different phases.
In~\cite{vescha02, decleitai09,vanasclei10}, a compact formulation \replyone{was} developed that allows $n$ level-set functions to represent up to $2^n$ distinct phases.
The use of multiple level-set functions complicates the interface representation, and significantly increases the computational cost for inverse problems.

In this paper, we propose a novel multilayer level-set method (MLSM) for eikonal-based first-arrival traveltime tomography.
The proposed MLSM uses a single level-set function to represent multiple subregions of distinct phases. Unlike conventional level-set methods that only utilize the zero-level set, the MLSM defines interfaces of multiple phases using a series of $i_n$-level-sets ($n=0,1,2\cdots$).
Near each $i_n$-level-set, the multilayer level-set function is designed to behave as a local signed-distance function.
This framework allows for the delineation of arbitrarily many distinct subregions using only one level-set function.
We then utilize the MLSM to reconstruct \replyone{multiphase} heterogeneous structures in traveltime tomography. 

The remainder of this paper is structured as follows.
Section 2 develops the formulation of the multilayer level-set function, and explains how to use it to represent piecewise structures with multiple interfaces.
Section 3 proposes the multilayer level-set method (MLSM) for eikonal-based traveltime tomography.
In this section, we adopt the Eulerian framework of~\cite{leuqia06,lileu13} to compute the Fr\'echet derivatives of the traveltime data misfit, develop comprehensive regularization techniques for the MLSM formulation, and propose an illumination-based error measure to evaluate the performance of the inversion algorithm.
Section 4 presents numerical results to demonstrate the algorithm's efficacy, and Section 5 draws conclusions.

\section{A novel multilayer level-set method (MLSM)}
\subsection{Formulation and interface representation}

The multilayer level-set method (MLSM) aims to represent interfaces of multiple phases by a single level-set function.
Considering a real function $\phi(\mathbf{x})$ and a series of real numbers $i_0<i_1<\cdots<i_{N-1}$ such that the following relation is satisfied.
 \begin{equation}\label{eqn2.1}
 \{ \phi(\mathbf{x})<i_0 \} \subset \{ \phi(\mathbf{x})<i_1 \} \subset \cdots \subset \{ \phi(\mathbf{x})< i_{N-1} \}.
 \end{equation}
This implies a multilayer structure, $\Omega_0\subset\Omega_1\subset\cdots\subset\Omega_{N-1}$, where 
\begin{equation}\label{eqn2.2}
\Omega_n=\{\mathbf{x}\mid\phi(\mathbf{x})<i_n\},\quad  \forall n=0,1,\cdots,N-1.
\end{equation}
Then the interface $\partial\Omega_n$ can be represented by the $i_n$-level-set,
\begin{equation}\label{eqn2.3}
\partial\Omega_n=\{\mathbf{x}\mid\phi(\mathbf{x})=i_n\},\quad \forall n=0,1,\cdots,N-1.
\end{equation}
With this formulation, a single level-set function can represent a series of interfaces of different phases.
It is natural to take the numbers $\{i_n\}_{n=0}^{N-1}$ as an arithmetic sequence, i.e., $\Delta i=i_n-i_{n-1}$ is constant; for example, $i_n=n$, or $i_n=\frac{n}{2}$ in a small computational domain.

Like the typical level-set method, the level-set function $\phi$ is expected to have a smooth transition through the interfaces $\partial\Omega_n$.
As a result, we design the level-set function $\phi(\mathbf{x})$ to behave as a signed-distance function near each $\partial\Omega_n$.
Let $d_n(\mathbf{x})$ denote the signed distance function to $\partial\Omega_n$,
\begin{equation}\label{eqn2.4}
d_n(\mathbf{x})=\left\{
\begin{array}{rl}
-\mathrm{dist}(\mathbf{x},\partial\Omega_n) \,,& \mathbf{x}\in\Omega_n \\
\mathrm{dist}(\mathbf{x},\partial\Omega_n) \,,& \mathbf{x}\in\Omega^c_n
\end{array}
\right. ,
\end{equation}
where $\mathrm{dist}(\mathbf{x},\partial\Omega_n):=\inf_{\tilde{\mathbf{x}}\in\partial\Omega_n}\|\mathbf{x}-\tilde{\mathbf{x}}\|_2$, $\forall n=0,1,\cdots,N-1$.
The multilayer level-set function is then designed in the following way:
\begin{equation}\label{eqn2.5}
\phi(\mathbf{x})=i_{n_\mathbf{x}}+\min\left(\max\left(d_{n_\mathbf{x}}(\mathbf{x}), -\frac{\Delta i}{2} \right),\frac{\Delta i}{2}\right)\,,\quad\mathrm{where}\ n_\mathbf{x}:=\mathrm{argmin}_n |d_n(\mathbf{x})|\,.
\end{equation}
With formula (\ref{eqn2.5}), \replyone{each} $\partial\Omega_n$ is represented by the $i_n$-level-set of $\phi(\mathbf{x})$; around $\partial\Omega_n$, $\phi(\mathbf{x})$ has the form of $\phi(\mathbf{x})=i_n+d_n(\mathbf{x})$.
Moreover, we impose a truncation on each signed-distance function $d_n(\mathbf{x})$, making it between $-\frac{\Delta i}{2}$ and $\frac{\Delta i}{2}$, $\forall n=0,1,\cdots,N-1$. Then the multilayer level-set function $\phi(\mathbf{x})$ lies in the region $\left[i_0-\frac{\Delta i}{2}, i_{N-1}+\frac{\Delta i}{2} \right]$. \replytwo{Note that $\phi(\mathbf{x})$ is not necessarily a continuous function; continuity is only required in the vicinity of each interface.}

\replytwo{Formula (\ref{eqn2.5}) together with (\ref{eqn2.4}) provides a robust framework for characterizing complex interface structures. This formulation naturally accommodates sequentially nested interfaces, as demonstrated in Example 1 (Section \ref{subsubsec_EX1} below). Furthermore, the multilayer level-set function can represent disjoint regions (see Example 2, Section \ref{subsubsec_EX2}); the key point is that $\Omega_n$ can be defined as either the interior or exterior of a region. Notably, while the level sets follow a strictly ordered relationship, for instance, $\{ \phi(\mathbf{x})<0 \} \subset \{ \phi(\mathbf{x})<1 \} \subset \{ \phi(\mathbf{x})< 2 \}$, the structure represented by the 0-level-set can appear to lie between those of the 1- and 2-level sets by appropriately defining $d_n(\mathbf{x})$ and $\phi(\mathbf{x})$; Example 3 in Section \ref{subsubsec_EX3} illustrates this.}

\replytwo{The following three examples provide concrete implementations of this formulation.}

\begin{figure}
 \raggedright
(a){{\includegraphics[scale=0.45,angle=0]{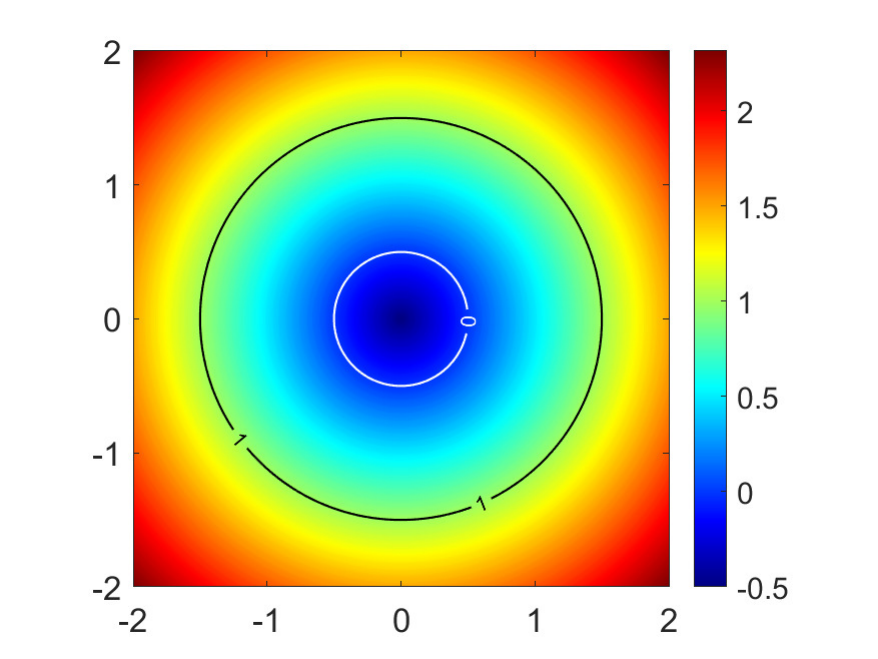}}}
(b){{\includegraphics[scale=0.45,angle=0]{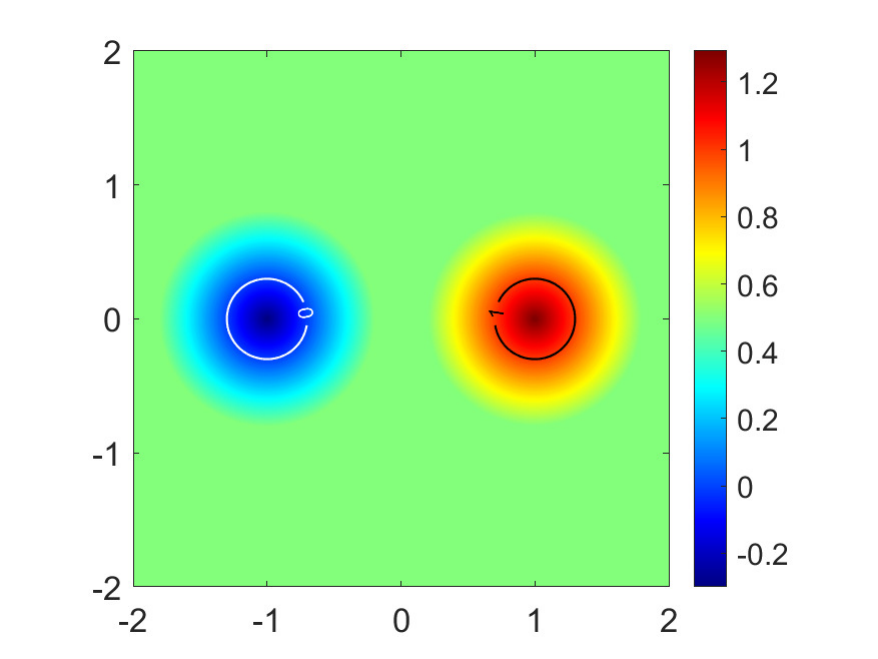}}}
\replytwo{(c){{\includegraphics[scale=0.45,angle=0]{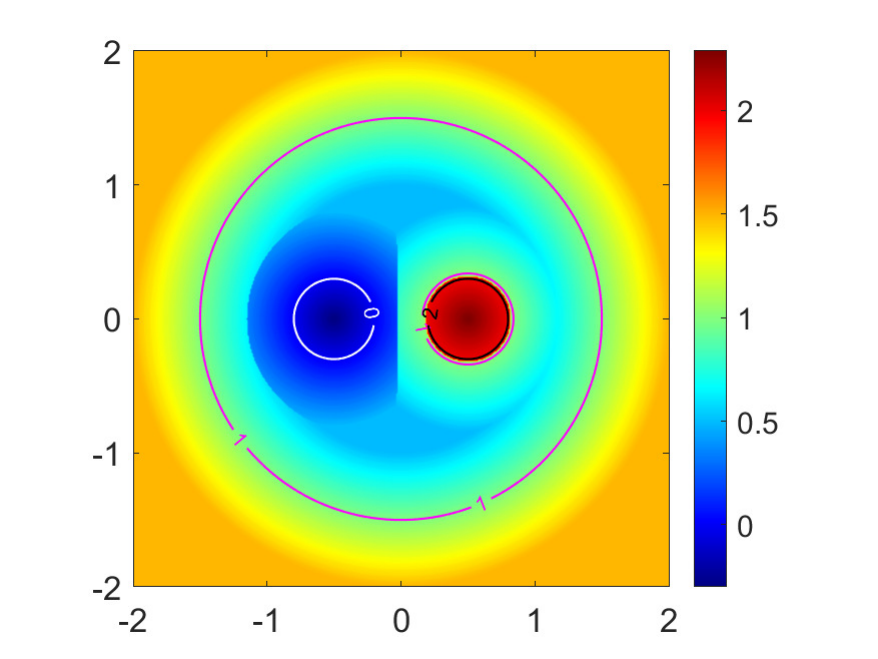}}}}
\caption{\replytwo{Three examples of the multilayer level-set function. (a) $\phi(x,y)$ from Section \ref{subsubsec_EX1}; (b) $\phi(x,y)$ from Section \ref{subsubsec_EX2}; (c) $\phi(x,y)$ from Section \ref{subsubsec_EX3}. In (a) and (b), the white line indicates the 0-level-set, and the black line indicates the 1-level-set; in (c), the white line indicates the 0-level-set, the magenta line indicates the 1-level-set, and the black line indicates the 2-level-set.}}
\label{Fig1}
\end{figure}

\subsubsection{Example 1: $\phi(x,y)=\sqrt{x^2+y^2}-0.5$.}\label{subsubsec_EX1}
It is the signed-distance function of the circle $\sqrt{x^2+y^2}=0.5$.
\replyone{By considering} $i_n=n$ and so $\Delta i=1$, it is also a \replyone{multilayer} level-set function satisfying equation (\ref{eqn2.5}).
Here, $\Omega_n=\{(x,y)\mid \sqrt{x^2+y^2}<n+0.5\}$, and $\partial\Omega_n=\{(x,y)\mid \sqrt{x^2+y^2}=n+0.5\}$. Figure \ref{Fig1}\,(a) plots $\phi(x,y)$ in $[-2,2]\times[-2,2]$, where we indicate its 0-level-set and 1-level-set, respectively.

\subsubsection{Example 2}\label{subsubsec_EX2}
Let $d_0(x,y)=\sqrt{(x+1)^2+y^2}-0.3$ and $d_1(x,y)=0.3-\sqrt{(x-1)^2+y^2}$. We define the level-set function,
\begin{equation*}
\phi(x,y)=\left\{
\begin{array}{rl}
\min\left(\max\left(d_0(x,y), -\frac{1}{2} \right),\frac{1}{2}\right)\,, & |d_0(x,y)|<|d_1(x,y)| \\
1+\min\left(\max\left(d_1(x,y), -\frac{1}{2} \right),\frac{1}{2}\right)\,, & |d_0(x,y)|\ge|d_1(x,y)|
\end{array}
\right.\,.
\end{equation*}
It is a multilayer level-set function in the form of equation (\ref{eqn2.5}). Here, $N=2$, $i_n=n,\, \forall n=0,1$, and so $\Delta i=1$.
Considering equations (\ref{eqn2.2}) and (\ref{eqn2.3}), the represented regions and interfaces are as follows:
\begin{eqnarray*}
\Omega_0=\{(x,y)\mid \sqrt{(x+1)^2+y^2}<0.3 \}, & \partial\Omega_0=\{(x,y)\mid \sqrt{(x+1)^2+y^2}=0.3\}\,; \\
\Omega_1=\{(x,y)\mid \sqrt{(x-1)^2+y^2}>0.3 \}, & \partial\Omega_1=\{(x,y)\mid \sqrt{(x-1)^2+y^2}=0.3\}\,.
\end{eqnarray*}
Figure \ref{Fig1}\,(b) plots this $\phi(x,y)$ in $[-2,2]\times[-2,2]$, where we indicate its 0-level-set and 1-level-set, respectively.

{
\subsubsection{Example 3}\label{subsubsec_EX3}
Let 
\begin{eqnarray*}
&&\Omega_0=\{(x,y)\mid \sqrt{(x+0.5)^2+y^2}<0.3 \}\,; \\
&&\Omega_1=\{(x,y)\mid \sqrt{(x-0.5)^2+y^2}>0.34\ \ \mathrm{and}\ \ \sqrt{x^2+y^2}<1.5 \}\,; \\
&&\Omega_2=\{(x,y)\mid \sqrt{(x-0.5)^2+y^2}>0.3 \}\,.
\end{eqnarray*}
According to (\ref{eqn2.4}), we define
\begin{eqnarray*}
&&d_0(x,y)= \sqrt{(x+0.5)^2+y^2}-0.3 \,; \\
&&d_1(x,y)=\max\left(0.34-\sqrt{(x-0.5)^2+y^2},\ \sqrt{x^2+y^2}-1.5\right)\,; \\
&&d_2(x,y)=0.3-\sqrt{(x-0.5)^2+y^2}\,.
\end{eqnarray*}
Then following equation (\ref{eqn2.5}), we have the multilayer level-set function,
\begin{equation*}
\phi(x,y)=\left\{
\begin{array}{rl}
\min\left(\max\left(d_0(x,y), -\frac{1}{2} \right),\frac{1}{2}\right)\,, & |d_0(x,y)|\le \min\left\{|d_1(x,y)|, |d_2(x,y)| \right\} \\
1+\min\left(\max\left(d_1(x,y), -\frac{1}{2} \right),\frac{1}{2}\right)\,, & |d_1(x,y)|\le \min\left\{|d_0(x,y)|, |d_2(x,y)| \right\} \\
2+\min\left(\max\left(d_2(x,y), -\frac{1}{2} \right),\frac{1}{2}\right)\,, & |d_2(x,y)|\le \min\left\{|d_0(x,y)|, |d_1(x,y)| \right\}
\end{array}
\right.\,.
\end{equation*}
Figure \ref{Fig1}\,(c) plots this $\phi(x,y)$ in $[-2,2]\times[-2,2]$, showing its 0-, 1-, and 2-level sets. In this example, the structure represented by the $0$-level-set appears to be positioned between those of the $1$- and $2$-level sets, despite the fact that the $0$-level-set does not cross the $1$-level-set to approach the $2$-level-set.
}

\subsection{Representing piecewise structures with multiple interfaces}
Consider a piecewise function as follows,
\begin{equation} \label{eqn2.6}
S(\mathbf{x})=\left\{
\begin{array}{lcl}
S_n (\mathbf{x}) & , & \mathbf{x}\in D_n,\ n=0,1,\cdots,N-1 \\
S_N (\mathbf{x}) & , & \mathbf{x}\in\left(\bigcup_{n=0}^{N-1} D_n\right) ^c
\end{array}
\right. \,.
\end{equation}
We explain how to represent $S(\mathbf{x})$ using the multilayer level-set function $\phi(\mathbf{x})$ proposed above.
\replyone{Consider} the definition of $\Omega_n$ in equation (\ref{eqn2.2}).
Let $D_0=\Omega_0=\{\mathbf{x} \mid \phi(\mathbf{x})<i_0\}$, and $D_n=\Omega_n\setminus\Omega_{n-1}=\{\mathbf{x} \mid i_{n-1}\le\phi(\mathbf{x})<i_n\}$, $\forall n=1,\cdots,N-1$.
The multilayer level-set representation of $S(\mathbf{x})$ reads as follows,
\begin{equation} \label{eqn2.7}
S(\mathbf{x})=\sum_{n=0}^{N-1} p_n\replytwo{(\mathbf{x})}\left(1-H(\phi\replytwo{(\mathbf{x})}-i_n) \right) + p_N\replytwo{(\mathbf{x})} H(\phi\replytwo{(\mathbf{x})}-i_{N-1})\,,
\end{equation}
where $H(\cdot)$ denotes the Heaviside function: $H(x)=\left\{\begin{array}{ccc}0 &, &x<0\\ 1 &,& x\ge 0  \end{array}\right.$.
In equation (\ref{eqn2.7}), the parameters $p_n\replytwo{(\mathbf{x})}$ are related to the function values $S_n\replytwo{(\mathbf{x})}$ in the following way,
\begin{equation}\label{eqn2.8}
\left\{
\begin{array}{ll}
S_n(\mathbf{x})=\sum_{k=n}^{N-1} p_k(\mathbf{x}) , & n=0,1,\cdots,N-1 \\
S_N(\mathbf{x})=p_N(\mathbf{x}) &
\end{array}
\right.\,,
\end{equation}
and so
\begin{equation}\label{eqn2.9}
\left\{
\begin{array}{ll}
p_n(\mathbf{x})=S_n(\mathbf{x})-S_{n+1}(\mathbf{x}) , & n=0,1,\cdots,N-2 \\
p_{N-1}(\mathbf{x})=S_{N-1}(\mathbf{x}) & \\
p_N(\mathbf{x})=S_N(\mathbf{x}) &
\end{array}
\right.\,.
\end{equation}
For example, when $N=2$, we have $p_0(\mathbf{x})=S_0(\mathbf{x})-S_1(\mathbf{x})$, $p_1(\mathbf{x})=S_1(\mathbf{x})$, and $p_2(\mathbf{x})=S_2(\mathbf{x})$.

\subsection{Reinitialization of the multilayer level-set function} \label{subsec2.3}
In a typical level-set method, reinitialization aims to maintain the signed-distance property of the level-set function near its 0-level set, which represents the interface structure.
Since the multilayer level-set function $\phi(\mathbf{x})$ represents a series of interfaces $\partial\Omega_n$ using its $i_n$-level-sets, $\forall n=0,1,\cdots, N-1$, the reinitialization is expected to maintain its signed-distance property around all these interfaces.

Firstly, solve the following reinitialization equation, $\forall n=0,1,\cdots, N-1$,
\begin{equation}\label{reinitialization}
\left\{
\begin{array}{l}\displaystyle
\frac{\partial\psi_n}{\partial\xi}+\mathrm{sign}(\phi-i_n)\,\left(|\nabla\psi_n|-1 \right)=0 \\
\psi_n(\mathbf{x},\xi=0)=\phi-i_n
\end{array}
\right.\,.
\end{equation}
$\psi_n(\mathbf{x}):=\psi_n(\mathbf{x},\xi=\infty)$ will be a signed-distance function to the $i_n$-level-set of $\phi(\mathbf{x})$.
In practice, because the reinitialization is repeatedly performed during the iterations of $\phi$, one only needs to evolve the equation for several $\Delta\xi$ steps, e.g. 5 steps.

Then, for every grid point $\mathbf{x}$, find 
\begin{equation}\label{eqn2.11}
n_\mathbf{x} =\mathrm{argmin}_n |\psi_n(\mathbf{x})|\,\replyone{,}
\end{equation}
\replyone{a}nd the multilayer level-set function is reinitialized as
\begin{equation}\label{eqn2.12}
\phi(\mathbf{x})=i_{n_\mathbf{x}}+\min\left(\max\left(\psi_{n_\mathbf{x}}(\mathbf{x}), -\frac{\Delta i}{2}\right),\frac{\Delta i}{2} \right)\,.
\end{equation}
Equations (\ref{eqn2.11}) and (\ref{eqn2.12}) exactly follow the definition of $\phi(\mathbf{x})$ in equation (\ref{eqn2.5}).

For example, when $N=2$, we have
\begin{equation*}
\phi(\mathbf{x})=\left\{
\begin{array}{rl}
i_0+\min\left(\max\left(\psi_0(\mathbf{x}), -\frac{\Delta i}{2} \right),\frac{\Delta i}{2}\right)\,, & |\psi_0(\mathbf{x})|<|\psi_1(\mathbf{x})| \\
i_1+\min\left(\max\left(\psi_1(\mathbf{x}), -\frac{\Delta i}{2} \right),\frac{\Delta i}{2}\right)\,, & |\psi_0(\mathbf{x})|\ge|\psi_1(\mathbf{x})| 
\end{array}
\right.\,.
\end{equation*}

\subsection{Some examples in moving interfaces}
Before addressing the inverse problems, we first demonstrate the performance of the MLSM through various representative examples of moving interfaces.
The interface evolution is governed by the motion law $\phi_t + v_n \lvert \nabla \phi \rvert = 0$, where \( v_n \) denotes the normal \replyone{speed} of the interface.
In these examples, we explore how multiple interfaces, represented within the MLSM framework, interact with each other when different motion laws are applied to different level sets of the same function.
These simple test cases highlight the effectiveness and ability of the MLSM.

\subsubsection{Motions in the normal direction}
\label{SubSubSec:NomalMotion}

\begin{figure}
\centering
\includegraphics[width=0.45\textwidth,angle=0]{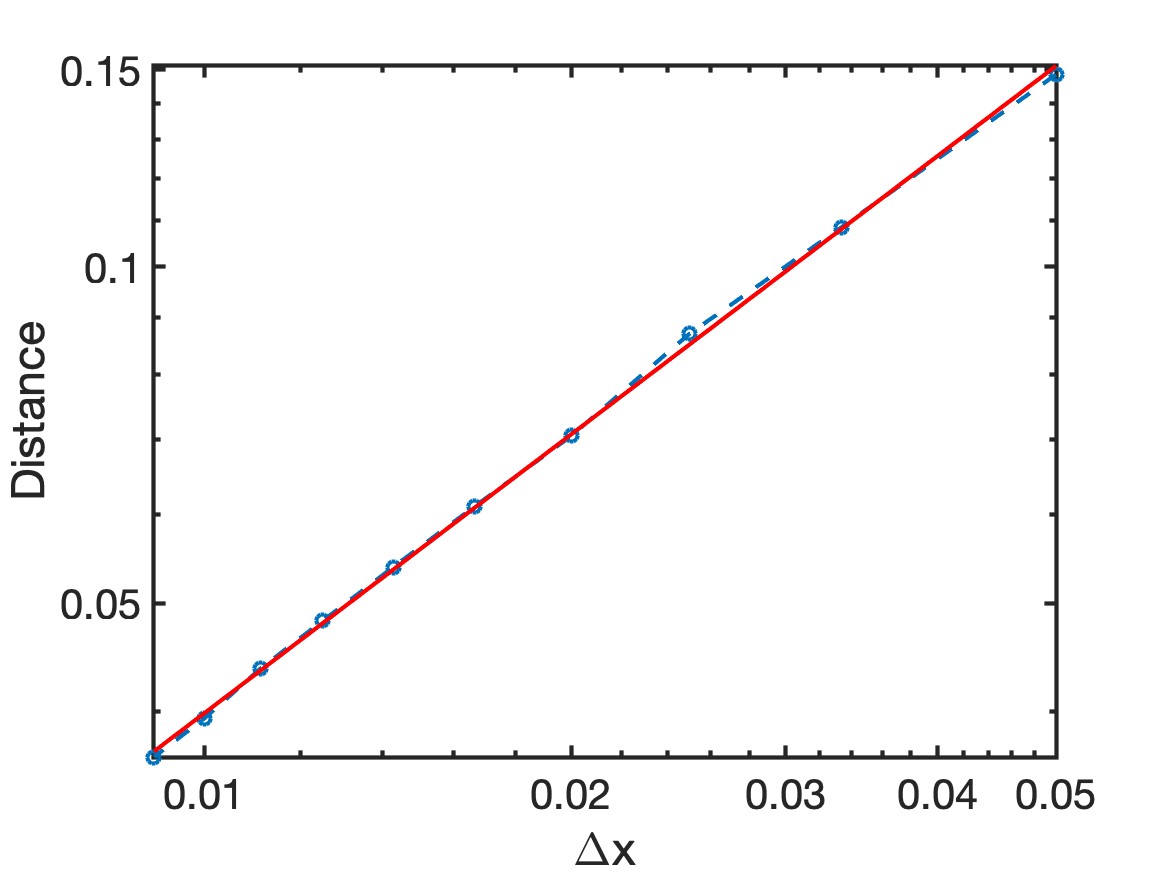}
\caption{(Section \ref{SubSubSec:NomalMotion} Test 1) The inner circle expands while the outer circle shrinks. The average distance between the two circles at the final time is computed using different mesh resolutions. The red line shows the least-squares fit, given by $1.8(\Delta x)^{0.83}$.
}
\label{Fig:NormalTwoCirclesFinal}
\end{figure}

\paragraph{Test 1} 
In this test, we consider the domain $[-2.5, 2.5]^2$ and initialize the multilayer level-set function with two concentric circles of radii $1$ and $2$, respectively.
The inner circle expands with a constant normal \replyone{speed} $v_n = 1$, while the outer circle shrinks at a constant \replyone{normal} speed $v_n = -1$.
Although the two interfaces meet at $t = 0.5$ and remain stationary thereafter, we evaluate the average distance between them at the final time $t_f = 1$.

The convergence plot in Figure~\ref{Fig:NormalTwoCirclesFinal} shows an almost linear trend.
As the mesh is refined, the average distance between the circles decreases.
This indicates that the MLSM captures the proximity of the interfaces more accurately with finer meshes, providing a more accurate representation of their behavior.

\begin{figure}
\centering
\includegraphics[width=0.99\textwidth,angle=0]{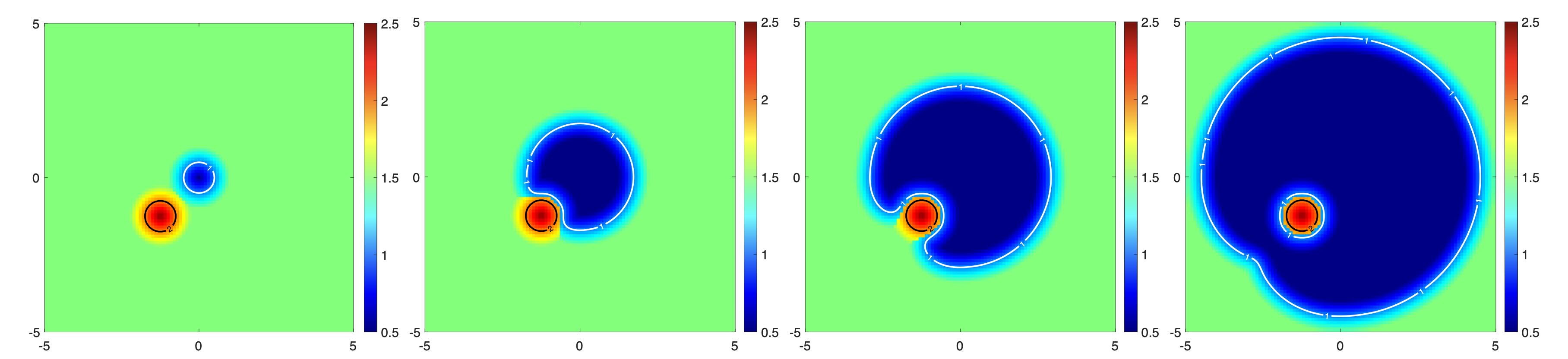}
\caption{(Section \ref{SubSubSec:NomalMotion} Test 2) The upper-right curve evolves according to motion in the normal direction, while the lower-left circle remains stationary. From left to right, the panels show the evolution of the solution using a mesh with $n = 101^2$ points.
}
\label{Fig:NormalTwoCircles101}
\end{figure}

\begin{figure}
\centering
\includegraphics[width=0.99\textwidth,angle=0]{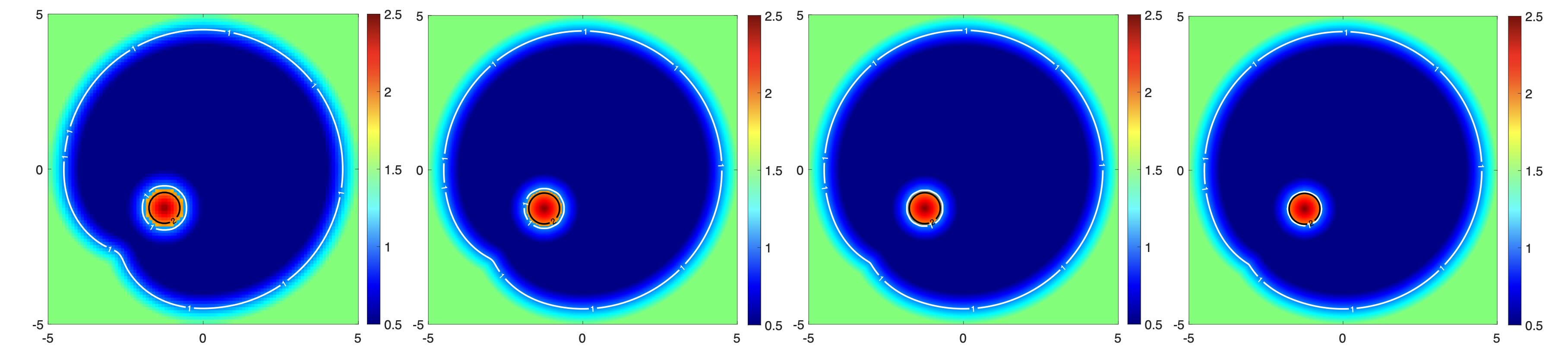}
\caption{(Section \ref{SubSubSec:NomalMotion} Test 2) The upper-right curve evolves according to motion in the normal direction, while the lower-left circle remains stationary. From left to right, the final solutions are shown using meshes with $n = 101^2$, $201^2$, $401^2$, and $801^2$ points, respectively.
}
\label{Fig:NormalTwoCirclesDiffn}
\end{figure}

\paragraph{Test 2} 

In many applications, moving interfaces interact with solid boundaries.
However, conventional level-set methods often encounter difficulties when updating the level-set function near boundaries.
To address this, one may introduce various extrapolation techniques to obtain ghost level-set values or normal vectors outside the computational domain~\cite{spe05,zahguskre09,gabvan24}, and develop new numerical schemes for reinitialization~\cite{zhaxuren22}.
In this example, we consider the interaction between a moving interface (represented by one level-set value in the MLSM framework) and a fixed solid boundary represented by another level-set value within the same \replyone{multilayer level-set} function.

In this test, we examine how interfaces behave as they approach each other.
The computational domain is $[-5,5]^2$, and the initial level-set function consists of two circles of radius $0.5$, centered at the origin and at $(-1.25,-1.25)$, respectively.
One circle is held fixed, while the other expands with a constant \replyone{normal} speed $v_n = 1$ until the final time $t_f = 4$.
Figure~\ref{Fig:NormalTwoCircles101} shows the computed solution using a mesh with $n = 101^2$ points.
The results demonstrate that the expanding circle does not merge with the obstacle; instead, a thin layer of interface \replyone{is formed} around the obstacle while the remaining parts of the circle continue to expand.
As shown in Figure~\ref{Fig:NormalTwoCirclesDiffn}, the thickness of this layer decreases as the mesh is refined. 
This behavior is consistent with our previous numerical experiment and illustrates the robustness of the MLSM when interfaces come into proximity.

\subsubsection{Motions by Mean Curvature}
\label{SubSubSec:MeanCurvature}

In this example, we consider motion\replyone{s} by mean curvature, in which higher-order derivatives of the level-set function appear in the governing equation.
This example tests the ability of the MLSM to handle more complex geometric flows, where numerical stability and accuracy become more challenging due to the involvement of higher-order spatial derivatives.

\begin{figure}
    \centering
    \includegraphics[width=0.99\textwidth,angle=0]{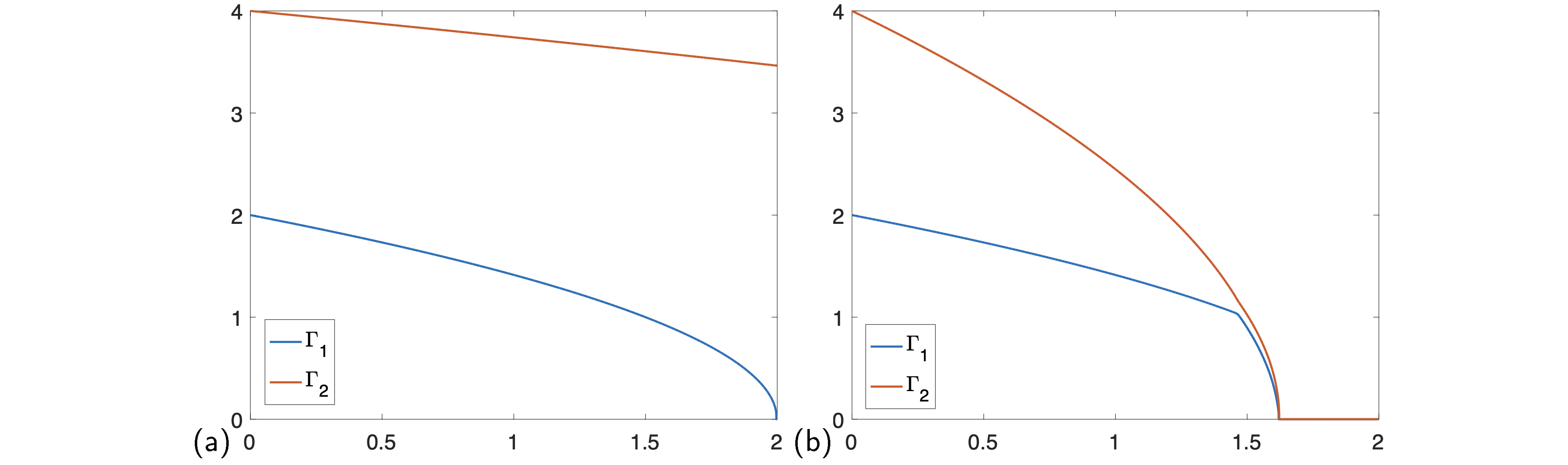}
\caption{(Section~\ref{SubSubSec:MeanCurvature}, Test~1) Time evolution of the mean radii of two circles, \( \Gamma_1 \) and \( \Gamma_2 \), with initial radii of 2 and 4, respectively. (a) In the first case, both circles evolve with normal velocity \( v_n = -\kappa \) using meshes with $401^2$ points. (b) In the second case, the normal velocity of the outer circle \( \Gamma_2 \) is increased to \( v_n = -5\kappa \), again using meshes with $401^2$ points.}
\label{Fig:MeanCurvature_TwoCircles_Shrink}
\end{figure}

\paragraph{Test 1} 
The multilayer level-set function is initialized with two concentric circles of radii $2$ and $4$ centered at the origin.
We first consider the case where both circles shrink with normal \replyone{speed} \( v_n = -\kappa \), where \( \kappa \) denotes the mean curvature.
The exact radius of a shrinking circle is given by \( r(t) = \sqrt{r_0^2 - 2t} \).
In particular, the inner circle will shrink to a point and disappear at \( t = \frac{1}{2}r_0^2 = 2 \).
The solution computed using a $401 \times 401$ mesh is shown in Figure~\ref{Fig:MeanCurvature_TwoCircles_Shrink}(a).
We then modify only \replyone{the normal speed of the outer circle} to be \( v_n = -5\kappa \), while keeping the inner interface moving with \( v_n = -\kappa \).
If the inner interface were absent, the radius of the outer interface would follow the exact solution \( r(t) = \sqrt{r_0^2 - 10t} \), which predicts that it would shrink to a point and disappear at \( t = \frac{1}{10}r_0^2 = 1.6 \).
However, when the inner interface is present, it impedes the inward motion of the outer interface as the two approach each other.
Our numerical results show that once the outer circle reaches the inner circle, they move together and eventually vanish simultaneously at \( t = 2 \), as shown in Figure~\ref{Fig:MeanCurvature_TwoCircles_Shrink}(b).
This behavior highlights an important feature of the MLSM that the motion of one interface can influence and alter the dynamics of another interface on a different level, especially when they come into contact.

\begin{figure}
    \centering
\includegraphics[width=0.99\textwidth,angle=0]{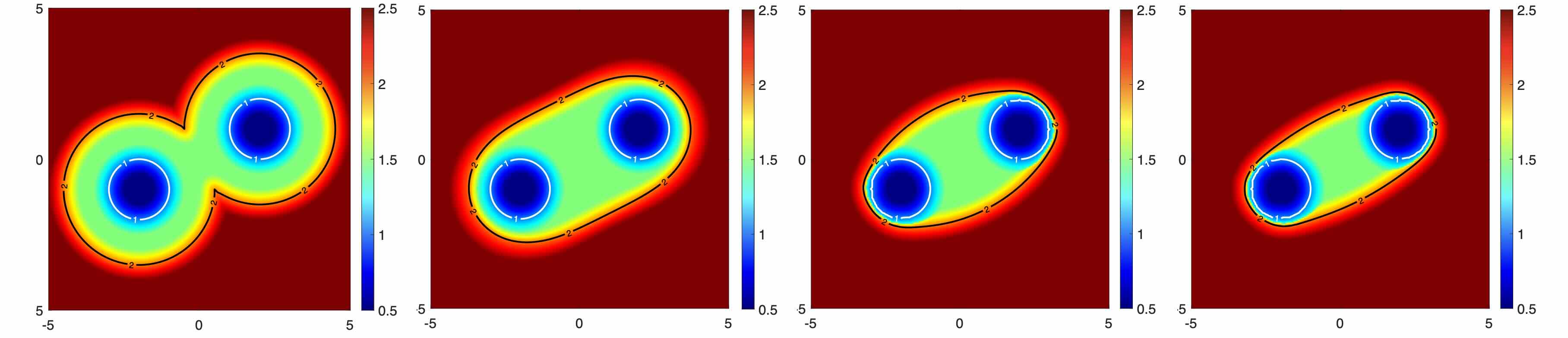}    
\caption{(Section~\ref{SubSubSec:MeanCurvature}, Test~2) The outer curve evolves according to mean curvature, while the inner circles remain stationary. The computation is performed using a mesh with $n = 201^2$ points.}
\label{Fig:MeanCurvature_TwoCircles_Obstacles}
\end{figure}

\paragraph{Test 2} 
In this case, the multilayer level-set function is initialized with an outer interface shaped like two merged circles and two inner circles of equal radius on the same level.
The outer interface evolves according to mean curvature given by \( v_n = -\kappa \), while the inner circles are kept fixed.
Figure~\ref{Fig:MeanCurvature_TwoCircles_Obstacles} shows the initial configuration and the results of the evolution.
As the outer interface evolves, it gradually deforms and adjusts its shape under the influence of mean curvature, while the inner circles remain stationary.
This experiment illustrates the capability of the MLSM to handle interfaces at different levels.
The outer curve undergoes curvature-driven motion, whereas the inner circles remain unaffected, preserving their initial configuration.

\section{Eikonal-based traveltime tomography}
Now we propose the multilayer level-set algorithm for eikonal-based traveltime tomography, where the multilayer level-set function is employed to represent discontinuous slowness structures with multiple phases and interfaces.

\subsection{Problem setup and level-set representation}
Consider the eikonal equation with a point-source boundary condition,
\begin{equation}\label{eqn3.1}
\left\{
\begin{array}{c}
|\nabla T(\mathbf{x})|=S(\mathbf{x})\,,\ \ \ \mathbf{x}\in\Omega \\
T(\mathbf{x}_s)=0
\end{array}
\right.
\end{equation}
where the viscosity solution $T(\mathbf{x})$ represents the first-arrival traveltime, $S(\mathbf{x})$ denotes the slowness parameter which is the reciprocal of wave velocity, and $\mathbf{x}_s$ denotes the location of point source.
The traveltime tomography reads as an inverse problem: Given the \replyone{measured} traveltime data $T^*$ on $\Gamma\subset\partial\Omega$, recover the slowness $S(\mathbf{x})$ in $\Omega$.

In practice, the measurement of traveltime is performed many times with a set of point sources.
Let $\{\mathbf{x}_{s,j}\}_{j=1}^J$ denote the set of point sources, and $T_j(\mathbf{x})$ be the corresponding traveltime,
\begin{equation}\label{eqn3.2}
\left\{
\begin{array}{c}
|\nabla T_j(\mathbf{x})|=S(\mathbf{x})\,,\ \ \ \mathbf{x}\in\Omega \\
T_j(\mathbf{x}_{s,j})=0
\end{array}
\right.\,,\quad j=1,\cdots,J\,.
\end{equation}
Then the inverse problem of traveltime tomography reads: Given the \replyone{measured} traveltime data $T_j^*$ ($j=1,\cdots,J$) on $\Gamma\subset\partial\Omega$, recover the slowness $S(\mathbf{x})$ in $\Omega$.

To reconstruct a discontinuous slowness structure with multiple phases and interfaces, we consider $S(\mathbf{x})$ to be a piecewise continuous function as shown in equation (\ref{eqn2.6}), where $S_n(\mathbf{x})$ is assumed to be continuous, $\forall n=0,\cdots,N$.
The multilayer level-set formulation is employed to represent $S(\mathbf{x})$, as shown in equation (\ref{eqn2.7}).
Here, we have the Fr\'echet derivatives of $S(\mathbf{x})$ in the following way,
\begin{align}
    \frac{\partial S}{\partial\phi}&= -\sum_{n=0}^{N-1} p_n\, \delta(\phi-i_n) + p_N\, \delta\left(\phi-i_{N-1}\right) \label{eqn3.3}\\
\frac{\partial S}{\partial p_n}&=1-H(\phi-i_n) ,\qquad n=0,1,\cdots N-1 \label{eqn3.4}\\
\frac{\partial S}{\partial p_N}&= H\left(\phi-i_{N-1}\right) \label{eqn3.5}
\end{align}
where $\delta(\cdot)$ denotes the Dirac delta function arising from the derivative of Heaviside. The inverse problem of traveltime tomography reduces to recovering the level-set function $\phi(\mathbf{x})$ and the parameters $p_n(\mathbf{x})$.

 \subsection{Adjoint state methods}
We use an $L^2$-norm misfit function to evaluate the discrepancy between \replyone{the} simulated and \replyone{the measured} traveltime data,
\begin{equation}\label{eqn3.6}
E=\frac{1}{2}\sum_{j=1}^J\int_{\Gamma}\left|T_j-T_j^*\right|^2\,\mathrm{d}s\,,
\end{equation}
where $T$ depends on $S$ and so the level-set function $\phi$ and parameters $p_n$.
The tomography problem is solved by considering an optimization problem,
\begin{equation}\label{eqn3.7}
(\phi(\mathbf{x}),p_0(\mathbf{x}),\cdots,p_N(\mathbf{x}))=\arg\min\, E\big(T\big(S((\phi(\mathbf{x}),p_0(\mathbf{x}),\cdots,p_N(\mathbf{x}))\big)\big)\,.
\end{equation}
We propose the gradient descent algorithm to solve (\ref{eqn3.7}), and like our previous works, we use an adjoint state method to calculate the Fr\'echet derivative $\frac{\partial E}{\partial S}$.
\replytwo{This kind of technique is also applied in the field of optimal design~\cite{allgrebihboggod23}.}
\begin{proposition} 
The Fr\'echet derivative $\frac{\partial E}{\partial S}$ is given by
\begin{equation}\label{eqn3.8}
\frac{\partial E}{\partial S}=\sum_{j=1}^J \lambda_j\, S\,,
\end{equation}
where $\lambda_j$ is the adjoint state variable satisfying the following adjoint state equation,
\begin{equation}\label{eqn3.9}
    \left\{
\begin{array}{rcl} \displaystyle
-\nabla\cdot(\lambda_j\nabla T_j)=0 &,& \mathrm{in}\ \Omega \\
\lambda_j\frac{\partial T_j}{\partial \mathbf{n}}=T_j-T_j^* &,& \mathrm{on}\ \Gamma \\
\lambda_j=0 &,& \mathrm{on}\ \partial\Omega\setminus\Gamma
\end{array}
\right.\,,\quad j=1,\cdots,J\,.
\end{equation}
\end{proposition}
We refer readers to our previous works~\cite{leuqia06,lileuqia14} for the derivations of~\eqref{eqn3.8} and~\eqref{eqn3.9}.
Moreover, we mention that the eikonal equations~\eqref{eqn3.2} and the adjoint state equations~\eqref{eqn3.9} are solved by the fast sweeping method with a Godunov upwind scheme~\cite{zha05,leuqia06}.

Combining equation (\ref{eqn3.8}) and equations (\ref{eqn3.3})-(\ref{eqn3.5}), we have
\begin{align}
    \frac{\partial E}{\partial\phi}&= \left(-\sum_{n=0}^{N-1} p_n\, \delta(\phi-i_n) + p_N\, \delta\left(\phi-i_{N-1}\right)\right) \cdot \sum_{j=1}^J \lambda_j\, S \label{eqn3.10}\\
    \frac{\partial E}{\partial p_n}&=\left(1-H(\phi-i_n)\right) \cdot \sum_{j=1}^J \lambda_j\, S, \qquad n=0,1,\cdots N-1 \label{eqn3.11}\\
    \frac{\partial E}{\partial p_N}&= H\left(\phi-i_{N-1}\right) \cdot \sum_{j=1}^J \lambda_j\, S \label{eqn3.12}
\end{align}
Numerically, we use a smooth version of the Heaviside function,
\begin{equation}\label{eqn3.13}
H_{\tau}(\phi)=\frac{1}{2}\left(\tanh\frac{\phi}{\tau} +1\right)\,,
\end{equation}
and the numerical Delta function is
\begin{equation}
\delta_{\tau}(\phi)=H_{\tau}'(\phi)=\frac{1}{2\tau\cosh^2\frac{\phi}{\tau}}\,.
\end{equation}
Then the level-set function $\phi(\mathbf{x})$ and the parameters $p_n(\mathbf{x})$ are updated according to their gradient descent directions.


\subsection{Regularization}

\subsubsection{Reinitialization of $\phi$}
Firstly, the multilayer level-set reinitialization is imposed on $\phi$ during its \replyone{update}.
As discussed in Section \ref{subsec2.3}, it helps to maintain the signed-distance property of $\phi$ near its $i_n$-level-sets.
This is a regularization on $\phi$ making it well-behaved around the \replyone{interfaces}.
In our multilayer level-set formulation, the reinitialization is performed according to equations (\ref{reinitialization})-(\ref{eqn2.12}).
Ideally, the reinitialization does not change the $i_n$-level-set of $\phi$, $\forall n=0,\cdots,N-1$, and so it will not affect the Fr\'echet derivatives of $E$ as shown in equations (\ref{eqn3.10})-(\ref{eqn3.12}).

\subsubsection{Penalization on $\phi$} \label{subsubsec_Penal}
A further penalization can be imposed on $\phi$ to regularize the shape of interfaces.
In \replyone{classical level-set methods}, the $0$-level-set represents the interface, and its arc-length or area can be evaluated by $\int_\Omega \delta(\phi) |\nabla\phi|\mathrm{d}\mathbf{x}$~\cite{zhachamerosh96}, where $\delta(\cdot)$ denotes the Dirac delta function.
In the multilayer level-set method, the interfaces are represented by $i_n$-level-sets as shown in equation (\ref{eqn2.3}).
As a result, we consider the following term,
\begin{equation}\label{penal1}
E_r(\phi)=\int_\Omega \sum_{n=0}^{N-1}\delta(\phi-i_n) |\nabla\phi|\mathrm{d}\mathbf{x}\,.
\end{equation}
Here, $E_r(\phi)$ evaluates the \replyone{sum of the lengths or areas} of all interfaces represented by the multilayer level-set function.
Penalizing $E_r(\phi)$ has the effect of shrinking the interfaces, and thus avoiding forming sharp corners in the shape of interfaces.
We add the penalization term $E_r(\phi)$ to the data misfit function $E$, and form the total-energy functional,
\begin{equation}\label{penal2}
E_{\mathrm{total}}=E+\gamma_\phi\, E_r(\phi)\,,
\end{equation}
where $E$ is defined in equation (\ref{eqn3.6}), and $\gamma_\phi>0$ is a parameter that controls the amount of penalization applied.
Using calculus of variations, the Fr\'echet derivative of $E_r(\phi)$ is,
\begin{equation}\label{penal3}
\frac{\partial E_r}{\partial\phi}=-\sum_{n=0}^{N-1}\delta(\phi-i_n)\, \nabla\cdot\left(\frac{\nabla\phi}{|\nabla\phi|} \right)\,.
\end{equation}
The Dirac delta function can be approximated as~\cite{zhachamerosh96}
\begin{equation}\label{penal4}
\delta(\phi-i_n)\doteq\chi(T_{\hat{\tau}, i_n})|\nabla\phi| \,,
\end{equation}
where $\chi(\cdot)$ is the indicator function,
\begin{equation}\label{penal5}
\chi(D)(\mathbf{x})=\left\{
\begin{array}{lcr}
1 &, & \mathbf{x}\in D\\
0 &, & \mathbf{x}\notin D
\end{array}
\right.\,,
\end{equation}
and $T_{\hat{\tau}, i_n}$ denotes a small tube containing the $i_n$-level-set,
\begin{equation}\label{penal6}
T_{\hat{\tau}, i_n}=\{\mathbf{x}:\ |\phi(\mathbf{x})-i_n|<\hat{\tau} \}\,.
\end{equation}
Substituting (\ref{penal4}) into (\ref{penal3}), we have
\begin{equation}\label{penal7}
\frac{\partial E_r}{\partial\phi}\doteq -\chi(T_{\hat{\tau}})\,|\nabla\phi|
\, \nabla\cdot\left(\frac{\nabla\phi}{|\nabla\phi|} \right)\,,
\end{equation}
where $T_{\hat{\tau}}$ denotes the union of all the tubes containing the $i_n$-level-sets,
\begin{equation}\label{penal8}
T_{\hat{\tau}}= \bigcup_{n=0}^{N-1} \{\mathbf{x}:\ |\phi(\mathbf{x})-i_n|<\hat{\tau} \}\,.
\end{equation}
A drawback of formula (\ref{penal7}) is that it consists of a complicated curvature term $\nabla\cdot\left(\frac{\nabla\phi}{|\nabla\phi|}\right)$.
Since the multilayer level-set reinitialization is performed on $\phi$, one can expect that $|\nabla\phi|\doteq 1$ near the $i_n$-level-sets.
Then formula (\ref{penal7}) is simplified to be
\begin{equation}\label{penal9}
\frac{\partial E_r}{\partial\phi}\doteq -\chi(T_{\hat{\tau}})\,\Delta\phi\,,
\end{equation}
And
\begin{equation}\label{penal10}
\frac{\partial E_{\mathrm{total}}}{\partial\phi}= \frac{\partial E}{\partial\phi}-\gamma_\phi \cdot \chi(T_{\hat{\tau}})\,\Delta\phi\,.
\end{equation}
The added penalization effectively introduces a viscosity term near the $i_n$-level-sets. 

We mention that one should be very careful when using this penalization.
Firstly, according to equation (\ref{eqn2.5}), $\phi$ can be discontinuous away from its $i_n$-level-sets, where computing $\Delta\phi$ may introduce blowup.
Therefore, the tube width $\hat{\tau}$ should be taken small; for example, we take $\hat{\tau}=3\tau$ in computations, where $\tau$ is the parameter of our numerical Heaviside function as shown in equation (\ref{eqn3.13}).
Secondly, one should keep in mind that the penalization has the effect of shrinking the interfaces. 
To prevent \replyone{overly smoothing the interfaces}, this penalization is used only when necessary.

\subsubsection{Regularization of $p_n$} \label{subsubregul3}
To update $p_n$ smoothly, $\forall n=0,\cdots,N$, we \replyone{apply} a Sobolev regularization on the perturbations of $p_n$.
\replytwo{Similar techniques can be found in~\cite{bur03,degou06,lileu13}; see also~\cite[p. 125-126]{mohpir10}.}
Denoting $\tilde{p}_n=-\frac{\partial E}{\partial p_n}$, we solve the following equation to regularize $\tilde{p}_n$,
\begin{equation}\label{eqn3.15}
\left\{
\begin{array}{l} \displaystyle
(I-\gamma\Delta)\tilde{P}_n(\mathbf{x})=\tilde{p}_n(\mathbf{x})\,,\ \ \mathbf{x}\in\Omega \\
\displaystyle \frac{\partial\tilde{P}_n}{\partial\mathbf{n}}\mid_{\partial\Omega}=0
\end{array}
\right.\,,
\end{equation}
where $\mathbf{n}$ denotes the outer normal direction of $\partial\Omega$, and $\gamma>0$ is a parameter controlling the amount of regularization.
Then $\tilde{P}_n$ is used to replace $\tilde{p}_n$ in the updating of $p_n$.
To illustrate the effect of (\ref{eqn3.15}), we include the following proposition, which corrects some minor issues of Theorem 3.2 in~\cite{liqia21}. \replytwo{While its proof is straightforward for experts in PDE analysis, we include it here to ensure the paper remains self-contained for readers from broader computational and geophysical research communities.}

\begin{proposition}
Let $\Omega$ be a bounded domain with boundary of class $C^{1,1}$ (or alternatively, a bounded convex domain), and $\tilde{p}_n\in L^2(\Omega)$. The functional 
\[
J(\tilde{P})=\|\tilde{P}-\tilde{p}_n\|^2_{L^2(\Omega)}+\gamma\|\nabla\tilde{P}\|^2_{L^2(\Omega)},\ \ (\gamma>0)
\]
has a unique minimizer in $H^1(\Omega)$; this minimizer, denoted $\tilde{P}_n$, satisfies equation (\ref{eqn3.15}).
\end{proposition}
\begin{proof}
\setstretch{1.5}
(1) $J(\tilde{P})$ has a unique minimizer in $H^1(\Omega)$.

Since $J(\tilde{P})$ is convex and G\^ateaux differentiable, in $H^1(\Omega)$ the minimizer of $J$ is equivalent to the solution of its Euler-Lagrange equation:
\begin{equation*}
\left.\frac{\mathrm{d}}{\mathrm{d}\epsilon} J(\tilde{P}+\epsilon v)\right|_{\epsilon=0}=0\ \iff  \ \int_\Omega\left(\tilde{P}-\tilde{p}_n\right)v\mathrm{d}\mathbf{x}+\gamma\int_{\Omega}\nabla\tilde{P}\cdot\nabla v\mathrm{d}\mathbf{x}=0\,,\ \forall v\in H^1(\Omega)\,.
\end{equation*}
Define the bilinear form $a: H^1(\Omega)\times H^1(\Omega)\to\mathbb{R}$, and the linear functional $L: H^1(\Omega)\to\mathbb{R}$ as follows,
\begin{equation*}
a(u,v)=\int_\Omega uv\mathrm{d}\mathbf{x}+\gamma\int_\Omega \nabla u\cdot\nabla v\mathrm{d}\mathbf{x}\,,\quad L(v)=\int_\Omega\tilde{p}_n v\mathrm{d}\mathbf{x}\,.
\end{equation*}
It is direct to verify that $a(u,v)$ is bounded and  coercive, and $L(v)$ is bounded; in fact, $a(u,v)\le(1+\gamma)\|u\|_{H^1}\|v\|_{H^1}$, $a(u,u)\ge\min\{1,\gamma\}\|u\|_{H^1}^2$, and $L(v)\le\|\tilde{p}_n\|_{L^2}\|v\|_{H^1}$. Then by the Lax-Milgram theorem, 
\begin{equation} \label{proof_eqn1}
\exists\mid u\in H^1(\Omega)\,,\quad \mathrm{such\ that}\ \ a(u,v)=L(v)\,,\ \forall v\in H^1(\Omega)\,.
\end{equation}
This $u$ is the unique minimizer of $J$ in $H^1(\Omega)$, denoted $\tilde{P}_n$.

(2) The minimizer $\tilde{P}_n\in H^2(\Omega)$, and satisfies equation (\ref{eqn3.15}).

By (\ref{proof_eqn1}), $\tilde{P}_n\in H^1(\Omega)$ satisfies
\begin{equation} \label{proof_eqn2}
\int_\Omega\tilde{P}_n v\mathrm{d}\mathbf{x}+\gamma\int_{\Omega}\nabla\tilde{P}_n\cdot\nabla v\mathrm{d}\mathbf{x}=\int_\Omega\tilde{p}_n v\mathrm{d}\mathbf{x}\,,\ \forall v\in H^1(\Omega)\,;
\end{equation}
it is a weak solution of the elliptic equation (\ref{eqn3.15}). Since $\Omega$ is a bounded domain with boundary of class $C^{1,1}$ (or a bounded convex domain), elliptic regularity theories, e.g. Theorem 2.2.2.5 and Theorem 3.2.1.3 in~\cite{grisvard85}, imply that $\tilde{P}_n\in H^2(\Omega)$. Then integration by parts leads to equation (\ref{eqn3.15}).
\end{proof}

Moreover, we mention that $\tilde{P}_n$ is a gradient descent direction of $p_n$.
\begin{align}
    E(p_n+\epsilon\tilde{P}_n)-E(p_n)&\doteq\int_{\Omega}\frac{\partial E}{\partial p_n}\,\epsilon\tilde{P}_n\mathrm{d}\mathbf{x}=-\epsilon\int_{\Omega}\tilde{p}_n\tilde{P}_n\mathrm{d}\mathbf{x}=-\epsilon\int_{\Omega}(I-\gamma\Delta)\tilde{P}_n\cdot\tilde{P}_n\mathrm{d}\mathbf{x} \nonumber\\
    &=-\epsilon\int_{\Omega}\left|\tilde{P}_n\right|^2\,\mathrm{d}\mathbf{x}+\epsilon\gamma\left(\int_\Omega
-\left|\nabla\tilde{P}_n\right|^2\mathrm{d}\mathbf{x}+ \int_{\partial\Omega}\frac{\partial\tilde{P}_n}{\partial \mathbf{n}}\tilde{P}_n\,\mathrm{d}s\right) \nonumber\\
    &=-\epsilon\int_{\Omega}\left(\left|\tilde{P}_n\right|^2+\gamma\left|\nabla\tilde{P}_n\right|^2\right)\mathrm{d}\mathbf{x} \le 0\,. \label{eqn3.16}
\end{align}

Now, we summarize the inversion algorithm as follows.

\begin{algorithm}[!htb]\sffamily
\caption{\sffamily \hspace{2pt} Multilayer level-set algorithm for eikonal-based traveltime tomography}\label{alg:traveltime_tomography}
\begin{algorithmic}[1]
\State Prescribe the number $N$ and the arithmetic sequence $\{i_n\}_{n=0}^{N-1}$.
\State Initialize the multilayer level-set function $\phi$ and the parameters $p_n$, $n=0,\cdots,N$.
\State Construct the slowness function $S(\mathbf{x})$ according to the multilayer level-set formulation (\ref{eqn2.7}).
\State Compute the simulated traveltime $T_j$, $j=1,\cdots,J$, by solving the eikonal equation (\ref{eqn3.2}), and evaluate $E$ according to (\ref{eqn3.6}).
\State Solve the adjoint state variable $\lambda_j$ from equation (\ref{eqn3.9}), $j=1,\cdots,J$.
\State Evaluate the Fr\'echet derivatives according to equations (\ref{eqn3.10})-(\ref{eqn3.12}) and equation (\ref{penal9}); set $\tilde{\phi}=-\partial E/\partial\phi+\gamma_\phi \cdot \chi(T_{\hat{\tau}})\,\Delta\phi$ and $\tilde{p}_n=-\partial E/\partial\tilde{p}_n$, $n=0,\cdots,N$.
\State Regularize $\tilde{p}_n$ according to equation (\ref{eqn3.15}), use $\tilde{P}_n$ to replace $\tilde{p}_n$.
\State Update the model parameters: $\phi^{\mathrm{new}}=\phi+\epsilon\tilde{\phi}$, $p_n^{\mathrm{new}}=p_n+\epsilon\tilde{P}_n$, $n=0,\cdots,N$.
\State Perform the multilayer level-set reinitialization for $\phi$ according to equations (\ref{reinitialization})-(\ref{eqn2.12}).
\State Go back to step 3 until the stopping criterion is achieved: $E<\epsilon_{stop}$, or the number of iterations exceeds the specified value.
\end{algorithmic}
\end{algorithm}

\subsection{An error measure considering illumination} \label{subsec_illu}
We propose a novel measure to evaluate the performance of the inversion algorithm for first-arrival traveltime tomography.
A direct index is the error between \replyone{the} recovered slowness $S(\mathbf{x})$ and \replyone{the} true value $S_{\mathrm{true}}(\mathbf{x})$, e.g. $|S-S_{\mathrm{true}}|$.
(In \replyone{all} test examples, the true values are known.) 
However, in first-arrival traveltime tomography, the rays passing through the detection area are not uniformly distributed.
For instance, rays bend away from regions with large slowness, and information from such low-illumination regions is poorly detected by boundary measurements; when evaluating the performance of \replyone{an} inversion algorithm, one should allow larger mistakes in the low-illumination regions, which is the nature of the first-arrival traveltime tomography.

Similar to what we did in~\cite{lileuqia14}, we propose a labeling function, $F(\mathbf{x};\mathbf{x}_s)$, to indicate the illumination of the location $\mathbf{x}\in\Omega$ by rays emanating from the point source $\mathbf{x}_s$.
Specifically, when a first-arrival ray emitted from $\mathbf{x}_s$ is received by a boundary measurement at $\Gamma\subset\partial\Omega$, we set $F(\mathbf{x};\mathbf{x}_s)\equiv 1$ along the ray; denoting  $\{\mathbf{x}: \mathbf{x}= \mathbf{x}(\xi) \}$ the arc-length parametrization of the ray, it holds that
\begin{equation}\label{eqn3.17}
\frac{\mathrm{d}F\left(\mathbf{x}(\xi);\mathbf{x}_s\right)}{\mathrm{d}\xi}=0\,,
\end{equation}
where $\frac{\mathrm{d}}{\mathrm{d}\xi}$ denotes the total derivative along $\mathbf{x}= \mathbf{x}(\xi) $. Equation (\ref{eqn3.17}) implies that,
\begin{equation}\label{eqn3.18}
\nabla F(\mathbf{x};\mathbf{x}_s)\cdot\frac{\mathrm{d}\mathbf{x}(\xi)}{\mathrm{d}\xi}=0\,.
\end{equation}
Considering the characteristics of the eikonal equation, we have the ray direction $\frac{\mathrm{d}\mathbf{x}(\xi)}{\mathrm{d}\xi}$ as follows,
\begin{equation}\label{eqn3.19}
\frac{\mathrm{d}\mathbf{x}(\xi)}{\mathrm{d}\xi}=\frac{\nabla T(\mathbf{x})}{S(\mathbf{x})}\,,
\end{equation}
where $T(\mathbf{x})=T(\mathbf{x};\mathbf{x}_s)$ is the first-arrival traveltime with the point source $\mathbf{x}_s$.
Substituting (\ref{eqn3.19}) into (\ref{eqn3.18}), the labeling function $F(\mathbf{x};\mathbf{x}_s)$ satisfies the following equation,
\begin{equation}\label{eqn3.20}
\nabla F(\mathbf{x};\mathbf{x}_s)\cdot\nabla T(\mathbf{x})=0\,.
\end{equation}
In practice, $F(\mathbf{x};\mathbf{x}_s)$ is computed by propagating its value at the measurement boundary $\Gamma\subset\partial\Omega$ back into the domain $\Omega$,
\begin{equation}\label{eqn3.21}
\left\{
\begin{array}{rcl}
-\nabla F(\mathbf{x};\mathbf{x}_s)\cdot\nabla T(\mathbf{x})=0 &,& \mathbf{x}\in\Omega \\
F(\mathbf{x};\mathbf{x}_s)=1 &,&  \mathbf{x}\in\Gamma \\
F(\mathbf{x};\mathbf{x}_s)=0 &,&  \mathbf{x}\in\partial\Omega\setminus\Gamma
\end{array}
\right.\,.
\end{equation}

\replytwo{Equation (\ref{eqn3.21}) represents an Eulerian formulation of (\ref{eqn3.17}), asserting that the labeling function $F$ remains constant along each ray. Because the ray direction $\nabla T$ can be discontinuous, it may introduce shocks into the advection equation. For instance, for a piecewise Lipschitz continuous slowness $S(\mathbf{x})$, the eikonal equation admits an extended viscosity solution that is Lipschitz continuous \cite{ost20, lileu13}; however, Lipschitz continuity does not guarantee $C^1$ regularity everywhere.  In the context of first-arrival traveltimes, discontinuities in $\nabla T$ arise at kinks where multiple rays converge with identical minimal traveltimes. These kinks manifest as shocks in the advection equation (\ref{eqn3.21}). To resolve this, we employ an upwind scheme that respects causality: information flows exclusively into the shock, preventing singularities from propagating outward. Specifically, we solve (\ref{eqn3.21}) using the fast sweeping method with a Godunov upwind discretization \cite{lileuqia14}, with $F(\mathbf{x};\mathbf{x}_s)$ initialized to zero throughout $\Omega$.}

For multiple point sources $\{\mathbf{x}_{s,j}\}_{j=1}^J$, denoting $T_j(\mathbf{x}):=T(\mathbf{x};\mathbf{x}_{s,j})$ the first-arrival traveltime corresponding to the point source $\mathbf{x}_{s,j}$, we introduce a set of labeling functions, $F_j(\mathbf{x}):=F(\mathbf{x};\mathbf{x}_{s,j})$, $j=1,\cdots,J$, according to the advection equation (\ref{eqn3.21}), i.e.,
\begin{equation}\label{eqn3.22}
\left\{
\begin{array}{rcl}
-\nabla F_j(\mathbf{x})\cdot\nabla T_j(\mathbf{x})=0 &,& \mathbf{x}\in\Omega \\
F_j(\mathbf{x})=1 &,&  \mathbf{x}\in\Gamma \\
F_j(\mathbf{x})=0 &,&  \mathbf{x}\in\partial\Omega\setminus\Gamma
\end{array}
\right.\,.
\end{equation}
Then the total illumination function is defined as follows,
\begin{equation}\label{eqn3.23}
F(\mathbf{x})=\frac{1}{J}\sum_{j=1}^J F_j(\mathbf{x})\,.
\end{equation}
In (\ref{eqn3.23}), we compute the total illumination at $\mathbf{x}$ by summing contributions from all point sources, and we normalize it to make $F(\mathbf{x})\in[0,1]$; an $F(\mathbf{x})$ close to $1$ means the location $\mathbf{x}$ is of high illumination. 

In test examples where the true model $S_{\mathrm{true}}$ is known, the total illumination $F(\mathbf{x})$ should be computed using the first-arrival traveltime $T(\mathbf{x})$ calculated from $S_{\mathrm{true}}$.
Lastly, we define the error considering illumination,
\begin{equation}\label{eqn3.24}
e_F(\mathbf{x}):=F(\mathbf{x})\,\big|S(\mathbf{x})-S_{\mathrm{true}}(\mathbf{x})\big|\,,
\end{equation}
where $F(\mathbf{x})$ is the total illumination function, $S(\mathbf{x})$ is the recovered slowness, and $S_{\mathrm{true}}(\mathbf{x})$ denotes its true value.
To summarize, $e_F(\mathbf{x})$ tolerates large mistakes $|S(\mathbf{x})-S_{\mathrm{true}}(\mathbf{x})|$ in regions with low illumination.

\section{Numerical results}
The following test examples utilize a computational domain $\Omega=(-1,1)\times(0,2)$. $\overline{\Omega}$ is discretized into a $129\times129$ grid, yielding a mesh size of $h=2/128$. We consider $J=14$ shots,  where the point sources $\{\mathbf{x}_{s,j}\}_{j=1}^{14}$ are located at $(-0.9,\ 0.1:0.3:1.9)$ and $(0.9,\ 0.1:0.3:1.9)$, respectively.
The traveltime measurements are distributed on the mesh grid along the full boundary of $\Omega$, i.e., $\Gamma=\partial\Omega$.

We employ the multilayer level-set formulation to represent piecewise continuous slowness structures with multiple phases, where the interfaces are depicted by the $i_n$-level-sets of $\phi(\mathbf{x})$; we take $i_n=\frac{n}{2}$ in the test examples, e.g., $i_0=0$ and $i_1=0.5$.
The parameter $\tau$ in the numerical Heaviside and Delta functions is taken as $\tau=10^{-2}$ in the computation.


\subsection{Recovering domains/interfaces only}
Firstly, we provide examples where the slowness values $S_n(\mathbf{x})$, $n=0,\cdots,N$, are given, and we recover the domains $D_n$, $n=0,\cdots,N-1$.
In the multilayer level-set formulation (\ref{eqn2.7}), the parameters $p_n\replytwo{(\mathbf{x})}$ are given by the values of $S_n\replytwo{(\mathbf{x})}$ according to equation (\ref{eqn2.9}), and we only need to recover the level-set function $\phi(\mathbf{x})$. For simplicity, we consider examples with $N=2$.

\subsubsection{Example 1}
Figure \ref{Fig_EX1_1}\,(a) shows the true model of slowness.
It is a layered structure of considerable interest in geophysical exploration, where the slowness values are given as $S_0(\mathbf{x})=0.5$, $S_1(\mathbf{x})=1.0$ and $S_2(\mathbf{x})=2.0$. We consider an initial guess of the multilayer level-set function as follows,
\begin{equation}\label{eqn4.1}
\phi_0(x,z)=\sqrt{x^2+(z-1)^2}-0.3\,,
\end{equation}
which is a signed-distance function to the circle $\sqrt{x^2+(z-1)^2}=0.3$; its $0$-level-set is at $\sqrt{x^2+(z-1)^2}=0.3$, and its $0.5$-level-set is at $\sqrt{x^2+(z-1)^2}=0.8$.
The parameters $p_n$ in the multilayer level-set formulation (\ref{eqn2.7}) are given by the values of $S_n$ according to equation (\ref{eqn2.9}).
Figure \ref{Fig_EX1to3_initial} plots the initial guess, where Figure \ref{Fig_EX1to3_initial}\,(a) shows the initial level-set function, and Figure \ref{Fig_EX1to3_initial}\,(b) shows the initial structure of slowness.

The multilayer level-set inversion algorithm is then performed with an iteration step-size of $\epsilon=2\times10^{-3}$.
The penalization on $\phi$ is not employed in this example, i.e. $\gamma_\phi=0$.
Figure \ref{Fig_EX1_1}\,(b) shows the recovered level-set function after 5000 iterations; Figure \ref{Fig_EX1_1}\,(c) plots the solution of slowness; Figure \ref{Fig_EX1_1}\,(d) plots the discrepancy between the recovered slowness and its true model.
It shows that the mistake primarily appears at the interface regions with discontinuities, while the solution almost perfectly recovers the true model.

\begin{figure}
\centering
\includegraphics[width=0.99\textwidth,angle=0]{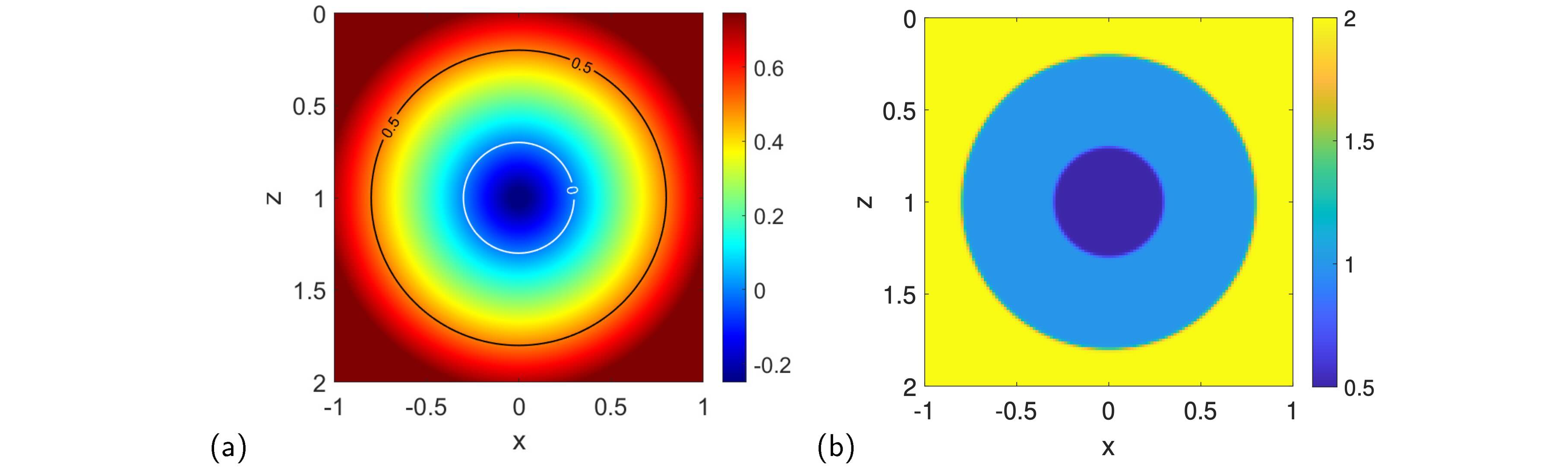}
\caption{Initial guess for Examples 1, 2 and 3. (a) Initial guess of the multilayer level-set function; (b) initial structure of slowness.}
\label{Fig_EX1to3_initial}
\end{figure}

\begin{figure}
\centering
\includegraphics[width=0.99\textwidth,angle=0]{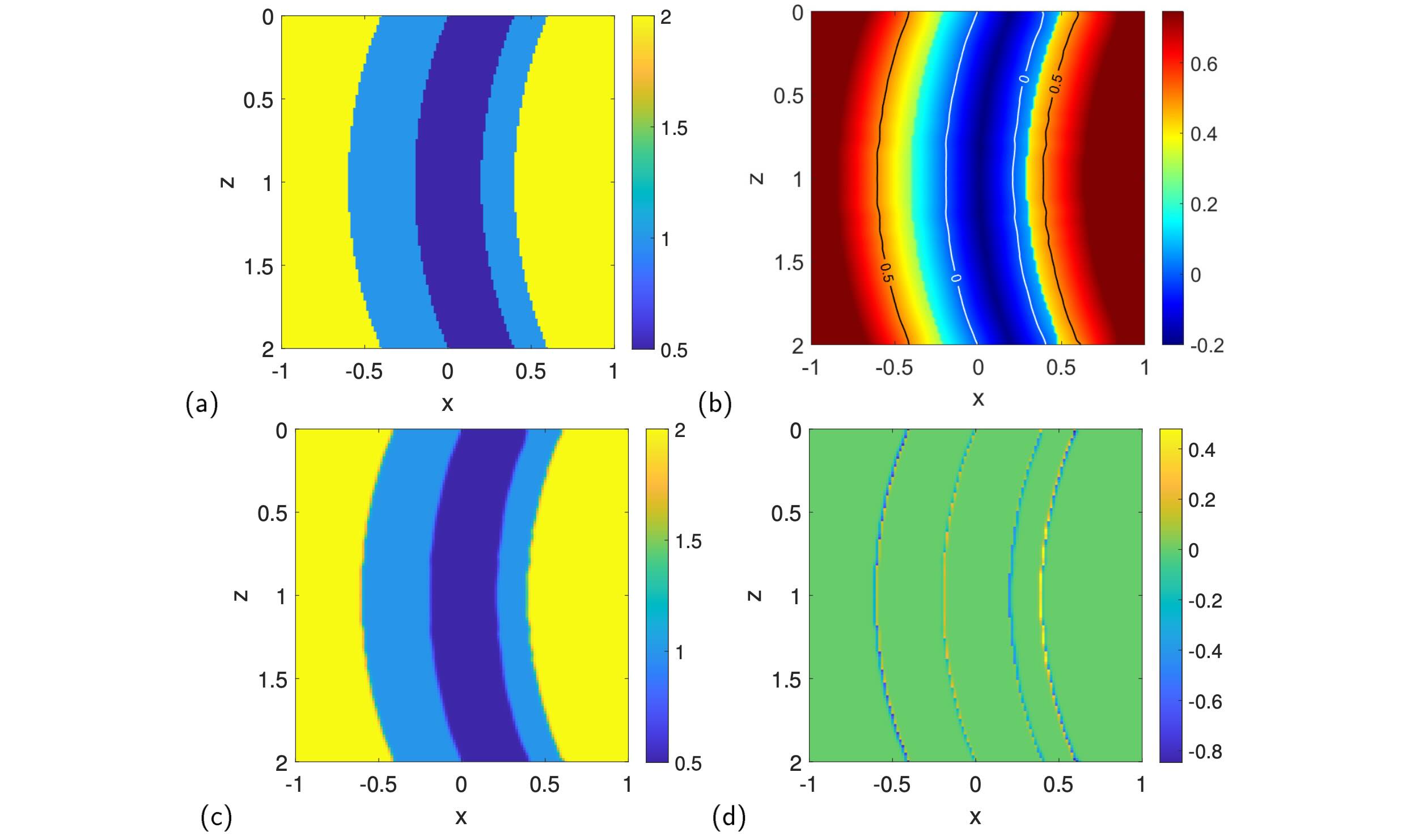}
\caption{Example 1: true model and recovered solution. The initial guess is shown in Figure \ref{Fig_EX1to3_initial}.  (a) True model of slowness $S_{\mathrm{true}}$; (b) recovered level-set function $\phi$; (c) recovered slowness $S$; (d) discrepancy $S-S_{\mathrm{true}}$.}
\label{Fig_EX1_1}
\end{figure}

\subsubsection{Example 2}
Figure \ref{Fig_EX2_1}\,(a) shows the true model of slowness.
It is a layered structure with a curved interface and a straight-line interface, where the slowness values are again given as $S_0(\mathbf{x})=0.5$, $S_1(\mathbf{x})=1.0$ and $S_2(\mathbf{x})=2.0$, respectively.
We consider an initial guess of $\phi(\mathbf{x})$ as shown in equation (\ref{eqn4.1}), which is the same as that of Example 1.
Figure \ref{Fig_EX1to3_initial} shows the initial guess.
The multilayer level-set inversion algorithm is then performed with an iteration step-size of $\epsilon=2\times10^{-3}$.
The penalization on $\phi$ is not employed in this example, i.e. $\gamma_\phi=0$.
Figure \ref{Fig_EX2_1}\,(b) shows the recovered level-set function.
Figure \ref{Fig_EX2_1}\,(c) plots the solution of slowness, and Figure \ref{Fig_EX2_1}\,(d) plots the discrepancy.
Again the mistake concentrates at the interface regions with discontinuities, and the solution almost perfectly recovers the true model.

\begin{figure}
\centering
\includegraphics[width=0.99\textwidth,angle=0]{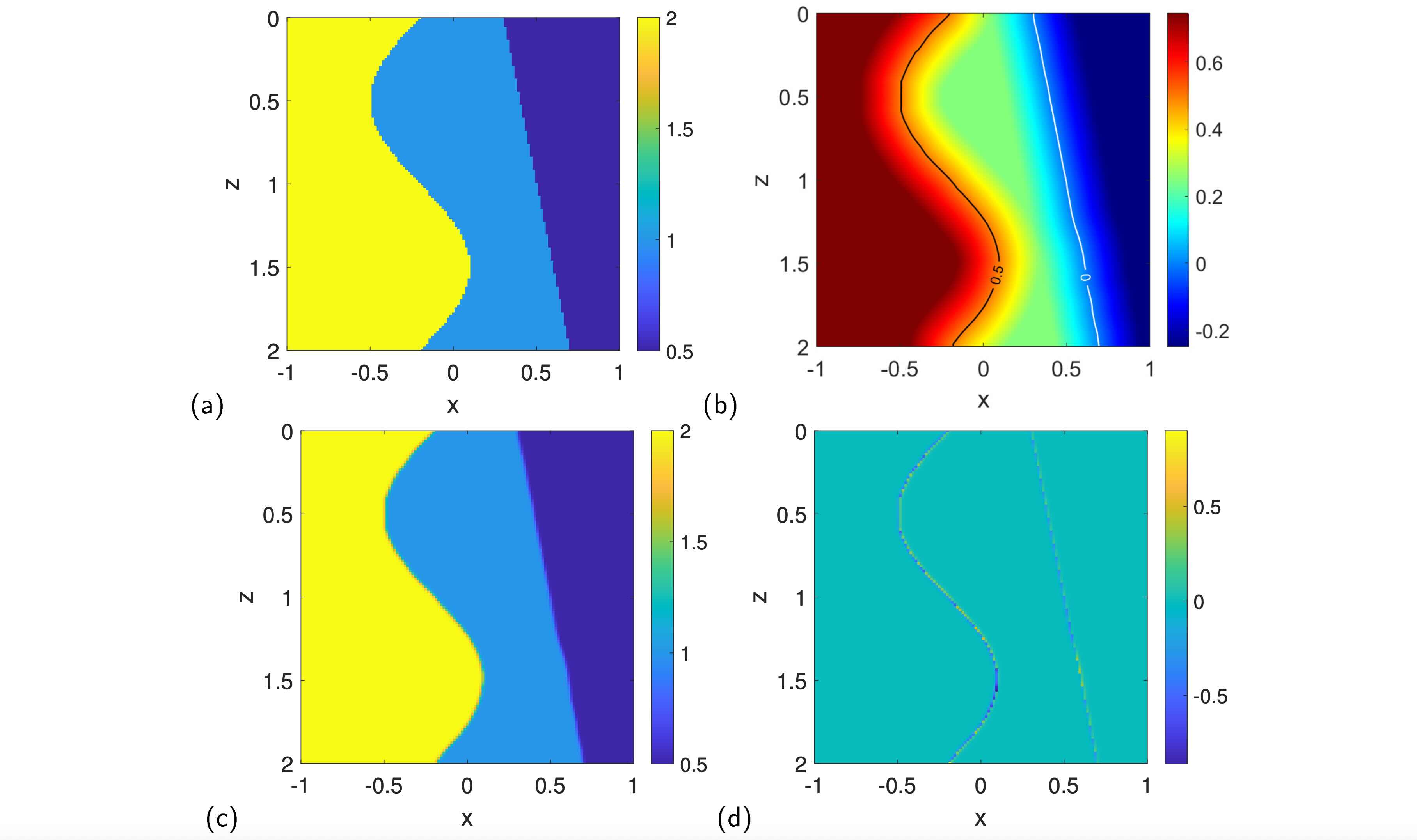}
\caption{Example 2: true model and recovered solution. The initial guess is shown in Figure \ref{Fig_EX1to3_initial}. (a) True model of slowness $S_{\mathrm{true}}$; (b) recovered level-set function $\phi$; (c) recovered slowness $S$; (d) discrepancy $S-S_{\mathrm{true}}$.}
\label{Fig_EX2_1}
\end{figure}

\subsubsection{Example 3}
Figure \ref{Fig_EX3_1} shows the true model of slowness.
A circle and a square, both with a slowness value of $S_0(\mathbf{x})=0.5$, are embedded within a rectangle that has a slowness value of $S_1(\mathbf{x})=1.0$.
The outermost region is homogeneous, with a slowness value of $S_2(\mathbf{x})=2.0$.
As before, we initialize the multilayer level set function $\phi(\mathbf{x})$ according to equation (\ref{eqn4.1}); Figure \ref{Fig_EX1to3_initial} plots the initial guess.
The multilayer level-set inversion algorithm is performed with an iteration step-size of $\epsilon=2\times10^{-3}$. 

To illustrate the effect of the penalization on $\phi$ as discussed in Section \ref{subsubsec_Penal}, we consider two inversions with $\gamma_\phi=0$ and $\gamma_\phi=0.01$, respectively.
Figure \ref{Fig_EX3_2} shows the inversion results, where the first row shows the recovered solution without the penalization on $\phi$, i.e. $\gamma_\phi=0$, and the second row shows the recovered solution with $\gamma_\phi=0.01$.
The inversion using the penalization on $\phi$ yields smoother interfaces in the solution. Given the complexity of the slowness model, both the two inversions provide satisfactory results.

\begin{figure}
\centering
\includegraphics[width=0.45\textwidth,angle=0]{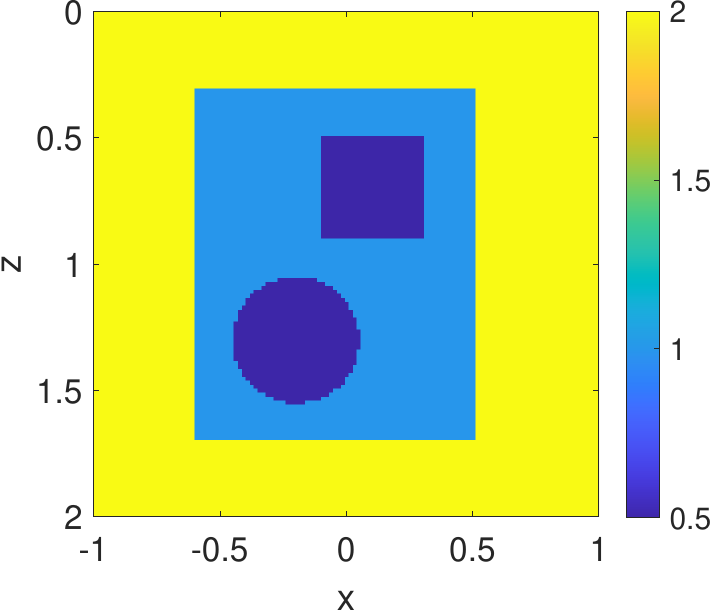}    
\caption{Example 3: true model $S_{\mathrm{true}}$.}
\label{Fig_EX3_1}
\end{figure}

\begin{figure}
\centering
\includegraphics[width=0.99\textwidth,angle=0]{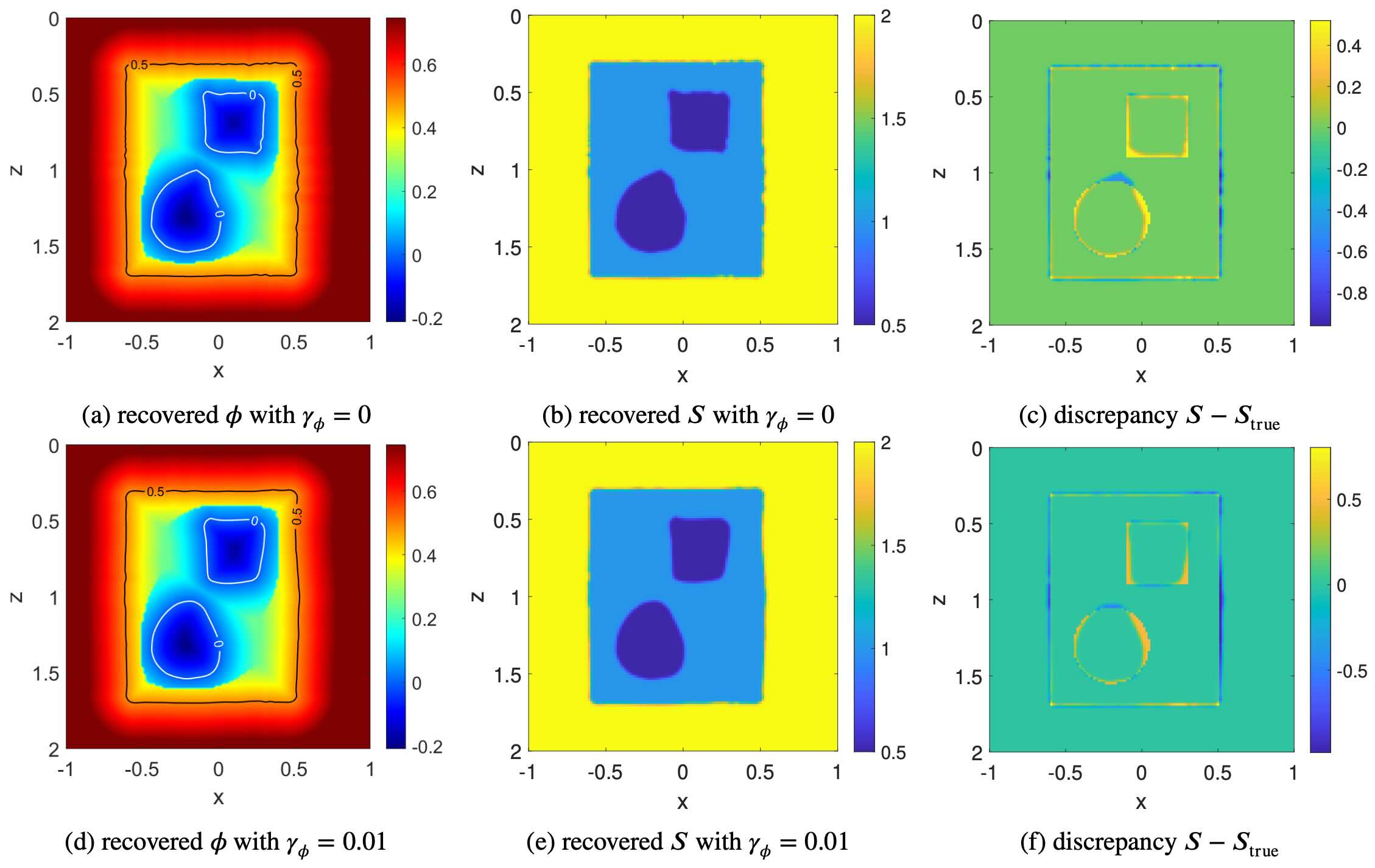}  
\caption{Example 3: recovered solutions with $\gamma_\phi=0$ and $\gamma_\phi=0.01$, respectively. The initial guess is shown in Figure \ref{Fig_EX1to3_initial}. The first row shows the inversion results without the penalization on $\phi$, i.e. $\gamma_\phi=0$. The second row shows the inversion results with $\gamma_\phi=0.01$.}
\label{Fig_EX3_2}
\end{figure}

\subsubsection{Example 4}
Figure \ref{Fig_EX4_1}\,(a) shows the true model of slowness, where the topology differs from that of examples 1-3.
The left triangle is defined by a non-constant slowness value $S_0(x,z)=z+0.1$; to its right, the homogeneous circle and square both have a slowness value of $S_2(\mathbf{x})=0.5$; the background region exhibits a slowness value of $S_1(\mathbf{x})=2.0$.

Considering the different topology, we initialize the multilayer level set function $\phi(\mathbf{x})$ in the following way,
\begin{equation} \label{eqn4.2}
\phi_0(x,z)=\left\{
\begin{array}{rl}
\min\left(\max\left(d_0(x,z), -0.25 \right), 0.25\right)\,, & |d_0(x,z)|<|d_1(x,z)| \\
0.5+\min\left(\max\left(d_1(x,z), -0.25 \right), 0.25\right)\,, & |d_0(x,z)|\ge|d_1(x,z)|
\end{array}
\right.\,,
\end{equation}
where
\begin{equation} \label{eqn4.3}
\left\{
\begin{array}{rl}
d_0(x,z)=\sqrt{(x+0.5)^2+(z-1)^2}-0.2 \,,\\
d_1(x,z)=0.2-\sqrt{(x-0.5)^2+(z-1)^2}\,.
\end{array}
\right.
\end{equation}
This initial level-set function $\phi_0(x,z)$ is similar to that in Figure \ref{Fig1}\,(b).
Consistently, the parameters $p_n$ in the multilayer level-set formulation are given by the values of $S_n$ according to equation (\ref{eqn2.9}).
Figure \ref{Fig_EX4_initial} plots the initial guess, where Figure \ref{Fig_EX4_initial}\,(a) plots the initial level-set function, and Figure \ref{Fig_EX4_initial}\,(b) plots the initial structure of slowness.

The multilayer level-set inversion algorithm is performed with an iteration step-size of $\epsilon=10^{-3}$.
The penalization on $\phi$ is not employed in this example, i.e. $\gamma_\phi=0$.
Figure \ref{Fig_EX4_1}\,(b) shows the recovered level-set function. Figure \ref{Fig_EX4_1}\,(c) plots the solution of slowness, and Figure \ref{Fig_EX4_1}\,(d) plots the discrepancy.
We conclude that the inversion result is quite good, considering the complexity of the slowness structure in this example.

\begin{figure}
\centering
\includegraphics[width=0.99\textwidth,angle=0]{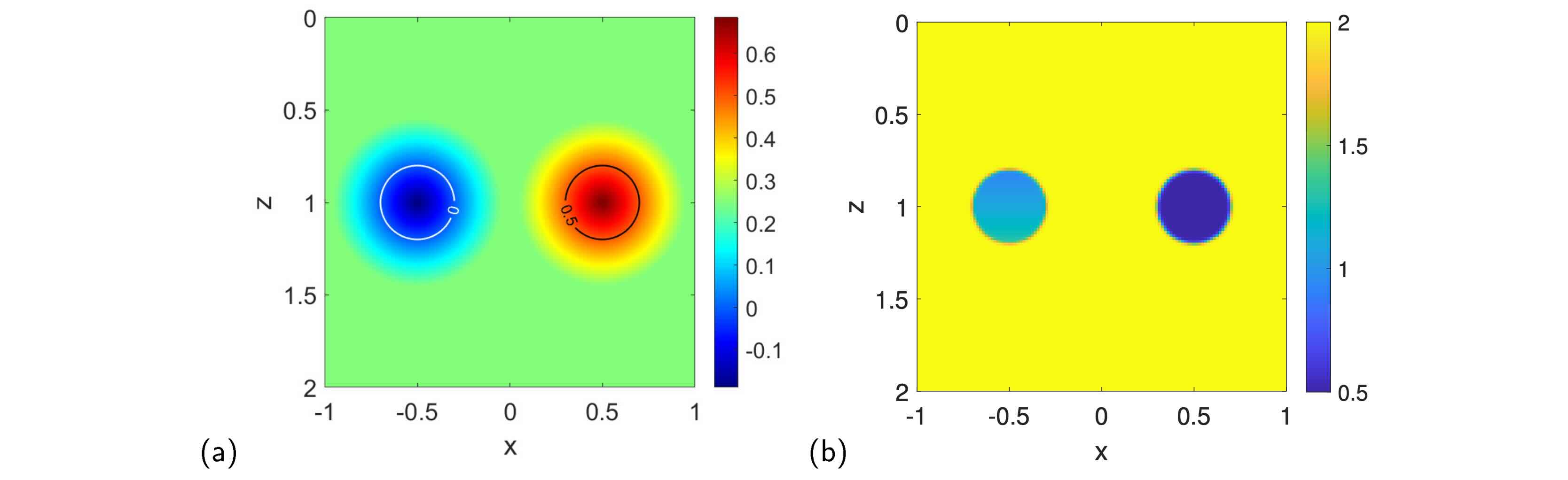}  
\caption{Initial guess for Example 4. (a) Initial guess of the multilayer level-set function; (b) initial structure of slowness.}
\label{Fig_EX4_initial}
\end{figure}

\begin{figure}
\centering
\includegraphics[width=0.99\textwidth,angle=0]{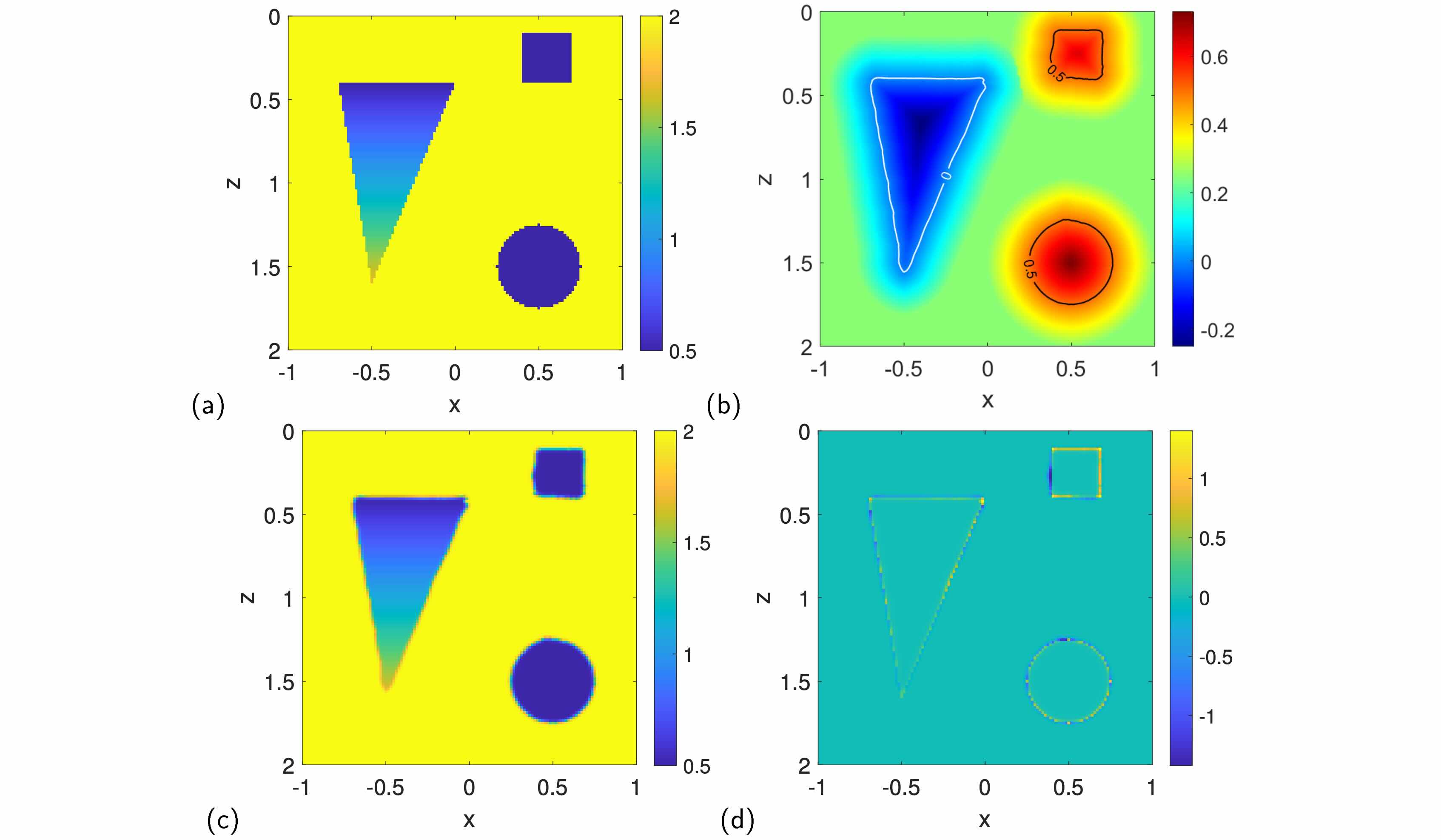}  
\caption{Example 4: true model and recovered solution. The initial guess is shown in Figure \ref{Fig_EX4_initial}. (a) True model of slowness $S_{\mathrm{true}}$; (b) recovered level-set function $\phi$; (c) recovered slowness $S$; (d) discrepancy $S-S_{\mathrm{true}}$.}
\label{Fig_EX4_1}
\end{figure}

\subsection{Recovering both domains/interfaces and slowness parameters}
In this part, we recover both the multilayer level-set function $\phi(\mathbf{x})$ and the slowness parameters $p_n\replytwo{(\mathbf{x})}$.
As discussed in Section \ref{subsubregul3}, the Sobolev regularization is performed on the \replyone{update} of $p_n\replytwo{(\mathbf{x})}$; here, the parameter $\gamma$ in equation (\ref{eqn3.15}) is taken as $\gamma=1$.

\subsubsection{Example 5}
The true model is the same as that of Example 1.
As shown in Figure \ref{Fig_EX5_1}\,(a), the slowness values of this layered structure are set as $S_0(\mathbf{x})=0.5$, $S_1(\mathbf{x})=1.0$ and $S_2(\mathbf{x})=2.0$.
We assume that the outermost slowness value $S_2(\mathbf{x})$ is given as a priori information, and recover the inner-region slowness functions $S_0(\mathbf{x})$ and $S_1(\mathbf{x})$; meanwhile, we reconstruct the interface locations.
According to equation (\ref{eqn2.7}), the slowness model is represented as
\begin{equation}
S(\mathbf{x})=p_0(\mathbf{x})\left(1-H(\phi\replytwo{(\mathbf{x})}) \right) +p_1(\mathbf{x})\left(1-H(\phi\replytwo{(\mathbf{x})}-0.5) \right) + p_2(\mathbf{x}) H\left(\phi\replytwo{(\mathbf{x})}-0.5\right)\,,
\end{equation}
where we set $N=2$ and $i_n=\frac{n}{2}$.
Considering equation (\ref{eqn2.9}), the true values of $p_n(\mathbf{x})$, ($n=0,1,2$) are $p_0(\mathbf{x})=S_0(\mathbf{x})-S_1(\mathbf{x})=-0.5$, $p_1(\mathbf{x})=S_1(\mathbf{x})=1.0$, and $p_2(\mathbf{x})=S_2(\mathbf{x})=2.0$.
Since the value of $S_2(\mathbf{x})$ is given as a priori information, we freeze $p_2(\mathbf{x})=2.0$, and reconstruct $p_0(\mathbf{x})$, $p_1(\mathbf{x})$, and the multilayer level-set function $\phi(\mathbf{x})$ simultaneously.

Figure \ref{Fig_EX5_6_initial} shows the initial guess for the inversion.
Here, the multilayer level-set function is initialized as equation (\ref{eqn4.1}), and the slowness functions are initialized as $S_0(\mathbf{x})=S_1(\mathbf{x})=1.2$, and so $p_0(\mathbf{x})=0$ and $p_1(\mathbf{x})=1.2$.
The initial guess is designed to consist of little information of the true model.
The multilayer level-set inversion algorithm is then performed with an iteration step-size of $\epsilon=2\times10^{-3}$.
The penalization on $\phi$ is not employed in this example, i.e. $\gamma_\phi=0$.
Figure \ref{Fig_EX5_1}\,(b) shows the recovered level-set function, Figure \ref{Fig_EX5_1}\,(c) shows the recovered slowness, and Figure \ref{Fig_EX5_1}\,(d) plots the discrepancy.
The solution closely matches the true model, where both the interfaces and slowness values are well recovered.

\begin{figure}
\centering
\includegraphics[width=0.99\textwidth,angle=0]{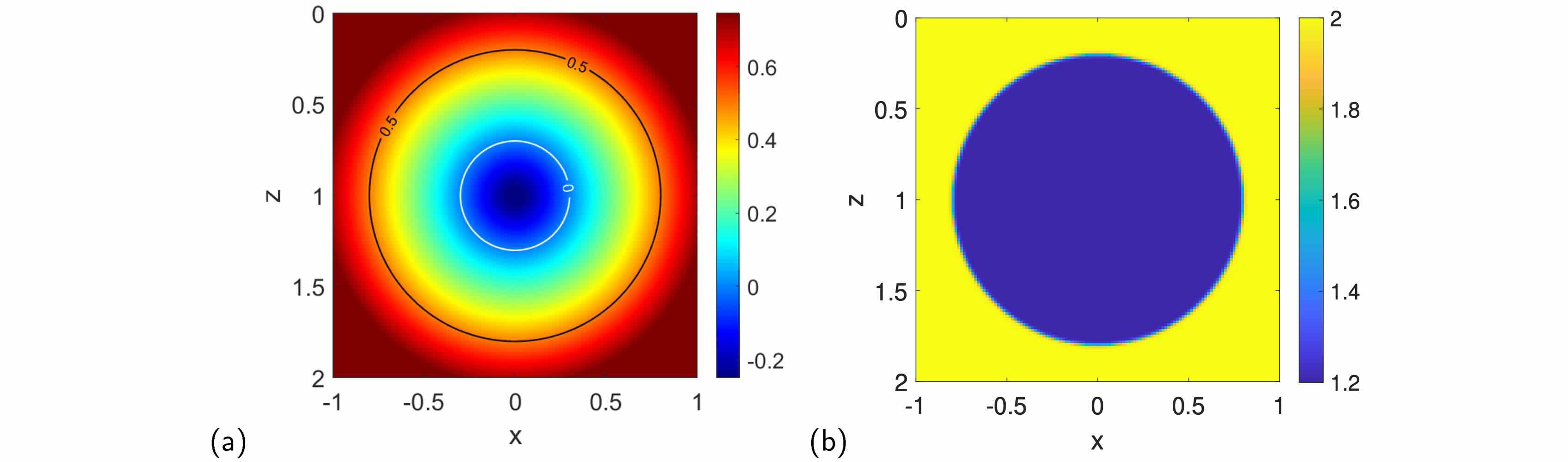}  
\caption{Initial guess for Example 5, Example 6, and Example 7. (a) Initial guess of the multilayer level-set function; (b) initial structure of slowness.}
\label{Fig_EX5_6_initial}
\end{figure}

\begin{figure}
\centering
\includegraphics[width=0.99\textwidth,angle=0]{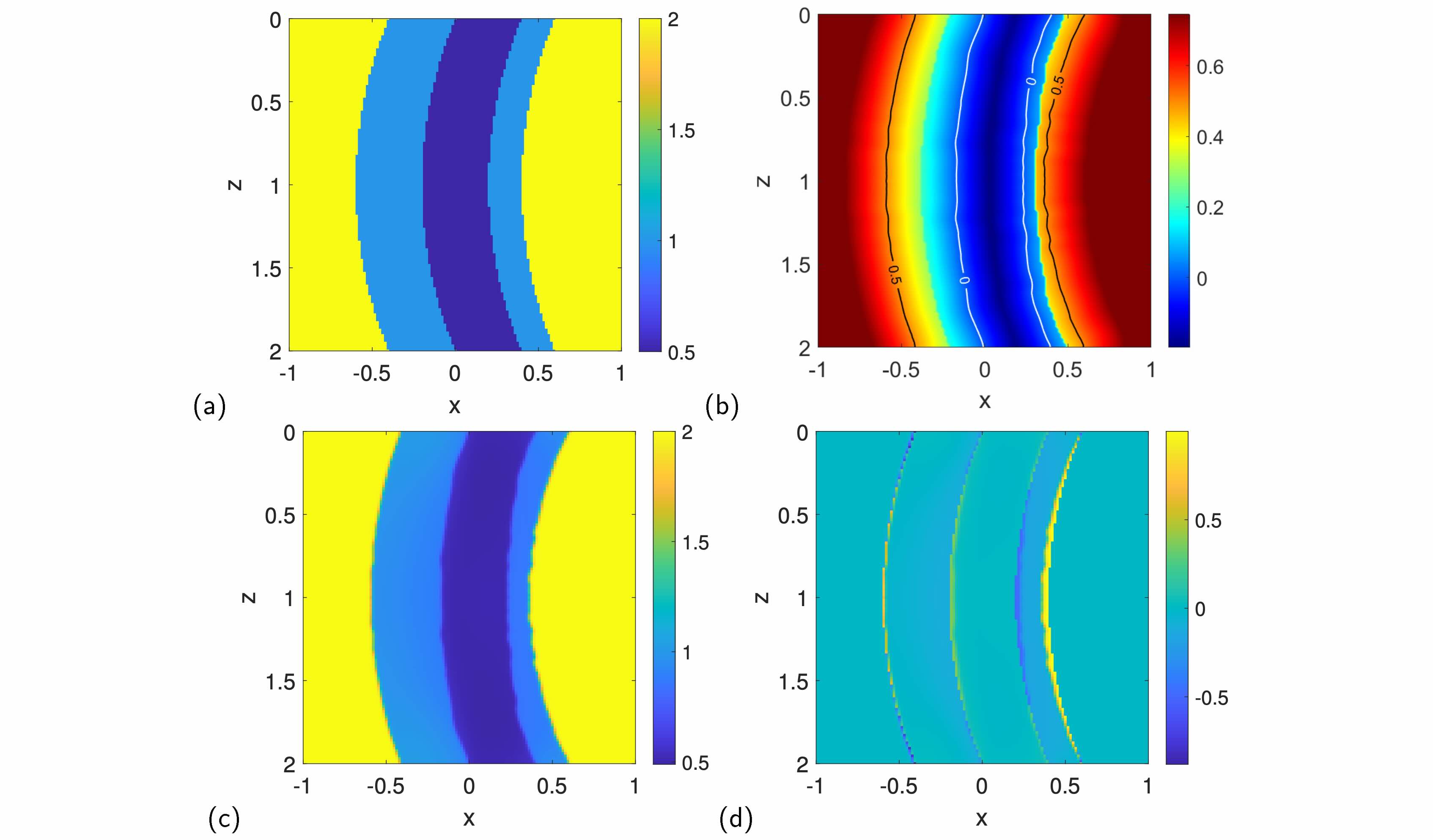}  
\caption{Example 5: recovering both domains and slowness parameters; true model and recovered solution. The initial guess is shown in Figure \ref{Fig_EX5_6_initial}. (a) True model of slowness $S_{\mathrm{true}}$; (b) recovered level-set function $\phi$; (c) recovered slowness $S$; (d) discrepancy $S-S_{\mathrm{true}}$.}
\label{Fig_EX5_1}
\end{figure}

\subsubsection{Example 6}
Figure \ref{Fig_EX6_1}\,(a) shows the true model.
The slowness values are set as $S_2(\mathbf{x})=2.0$, $S_1(\mathbf{x})=1.0$, and $S_0(\mathbf{x})=0.5 \exp\left\{2\left((x+0.3)^2+(z-1)^2)\right)\right\}$, respectively; the inner-region ellipse has a non-constant slowness.
Again we assume that the outermost slowness value $S_2(\mathbf{x})=2.0$ is given as a priori information, and reconstruct the inner-region slowness functions as well as the interface locations. 
The initial guess for the inversion is the same as that of Example 5, which is shown in Figure \ref{Fig_EX5_6_initial}.
The multilayer level-set inversion  is performed without the penalization term $E_r(\phi)$, i.e. $\gamma_\phi=0$.
Figure \ref{Fig_EX6_1}\,(b) shows the recovered multilayer level-set function, Figure \ref{Fig_EX6_1}\,(c) shows the recovered slowness, and Figure \ref{Fig_EX6_1}\,(d) plots the discrepancy.
We conclude that the solution matches well with the true model, where the structure with non-constant slowness is successfully recovered.

\begin{figure}
\centering
\includegraphics[width=0.99\textwidth,angle=0]{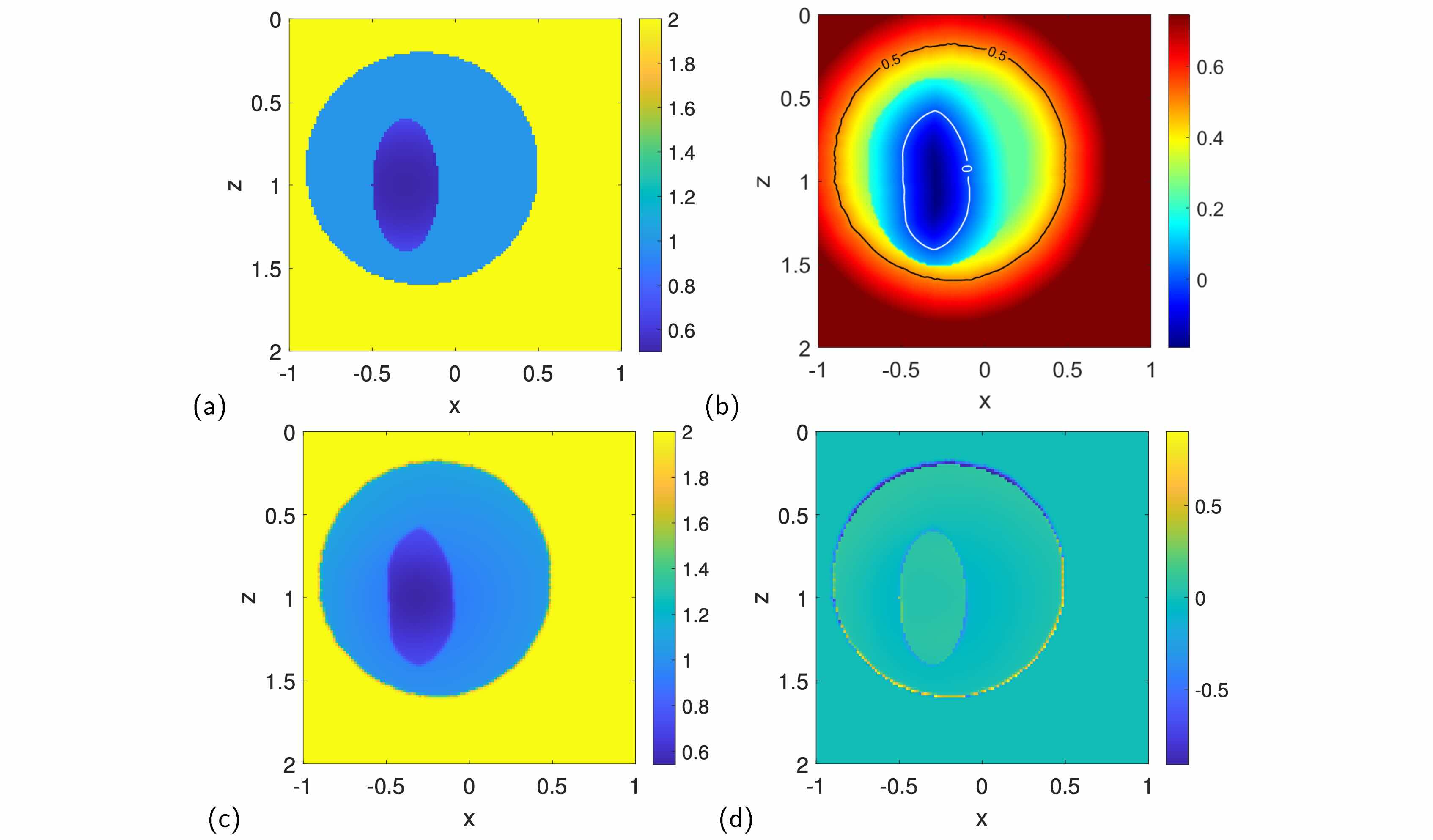}  
\caption{Example 6: recovering both domains and slowness parameters; true model and recovered solution. The initial guess is shown in Figure \ref{Fig_EX5_6_initial}. (a) True model of slowness $S_{\mathrm{true}}$; (b) recovered level-set function $\phi$; (c) recovered slowness $S$; (d) discrepancy $S-S_{\mathrm{true}}$.}
\label{Fig_EX6_1}
\end{figure}

\subsubsection{Example 7}
We provide an example to illustrate the error measure \replyone{which considers} illumination.
Figure \ref{Fig_EX7}\,(a) shows the true model.
The background region has a constant slowness of $S_2(\mathbf{x})=2.0$, the elliptic region is defined with a slowness of $S_1(\mathbf{x})=1.5 \exp\left\{-\left((x-0.1)^2+(z-1)^2)\right)\right\}$, and the inner rectangular region has a slowness of $S_0(\mathbf{x})=3.0$.
As before, we assume that the outermost slowness value, $S_2(\mathbf{x})=2.0$, is known a priori.
The objective is to reconstruct the inner-region slowness functions and delineate the interface locations.

The initial guess for the inversion is the same as that of Example 5, which is shown in Figure \ref{Fig_EX5_6_initial}.
The penalization on $\phi$ with a weight of $\gamma_\phi=0.01$ is applied in this example.
Figure \ref{Fig_EX7}\,(b) shows the recovered level-set function, Figure \ref{Fig_EX7}\,(c) shows the solution of slowness, and Figure \ref{Fig_EX7}\,(d) plots the discrepancy.
A noticeable imperfection in the solution occurs in the inner region with large slowness, which is likely attributable to a lack of illumination.
To account for this, we employ the error measure that considers illumination, as described in Section \ref{subsec_illu}. 
We first calculate the total illumination function $F(\mathbf{x})$ from the true model $S_\mathrm{true}$.
The error $e_F(\mathbf{x})$ is then computed according to equation (\ref{eqn3.24}). 
The resulting $F(\mathbf{x})$ and $e_F(\mathbf{x})$ are plotted in Figures \ref{Fig_EX7}\,(e) and \ref{Fig_EX7}\,(f), respectively.
As seen in Figure \ref{Fig_EX7}\,\replyone{(f)}, the error $e_F$ is primarily concentrated in regions with sharp interfaces, and the large discrepancy due to poor illumination is appropriately tolerated.
This error measure provides a more equitable assessment of the inversion algorithm's performance.

\begin{figure}
\centering
\includegraphics[width=0.75\textwidth,angle=0]{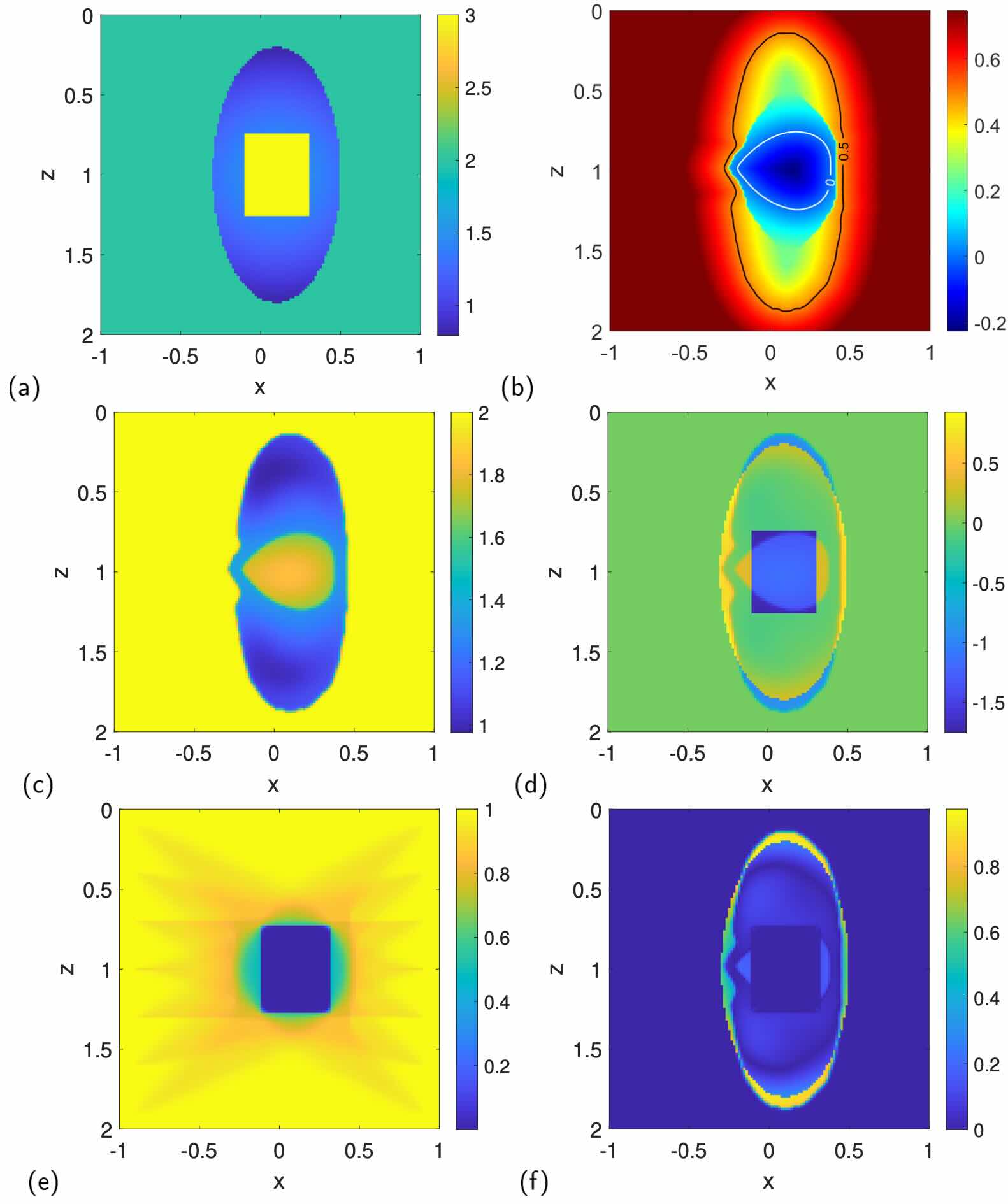}  
\caption{Example 7: recovering both domains and slowness parameters. The initial guess is shown in Figure \ref{Fig_EX5_6_initial}. (a) True model of slowness $S_{\mathrm{true}}$; (b) recovered level-set function $\phi$; (c) recovered slowness $S$; (d) discrepancy $S-S_{\mathrm{true}}$; (e) total illumination $F(\mathbf{x})$ for the slowness model $S_\mathrm{true}$; (f) error measure considering illumination $e_F(\mathbf{x})$.}
\label{Fig_EX7}
\end{figure}

\section{Conclusions}
We have developed a multilayer level-set method (MLSM) for eikonal-based first-arrival traveltime tomography.
In contrast to conventional level-set methods, which use only the zero-level set for interface representation, the MLSM defines the interfaces of multiple phases using a series of $i_n$-level-sets. 
Near each $i_n$-level-set, the multilayer level-set function is designed to behave as a local signed-distance function. Within this framework, a single level-set function can delineate arbitrarily many distinct interfaces and \replyone{subregions}.
We have detailed the representation of piecewise structures with multiple interfaces using the MLSM, and developed a reinitialization strategy for the \replyone{multilayer} level-set function.
After demonstrating the performance of the MLSM through various representative examples of moving interfaces, we propose it for eikonal-based first-arrival traveltime tomography.
The multilayer level-set function is employed to represent discontinuous slowness structures with multiple phases and interfaces.
To solve the tomography problem, we adopt an Eulerian framework, modeling the first-arrival traveltime via the viscosity solution of the eikonal equation, and calculating the Fr\'echet derivatives of the data misfit by the adjoint state method.
Comprehensive regularization techniques are proposed for the MLSM formulation, including reinitialization, arc-length penalization of the multilayer level-set function, and a Sobolev regularization for updating the MLSM parameters.
In addition, we introduce an illumination-based error measure to evaluate the performance of the inversion algorithm.
Extensive numerical examples illustrate the method's efficacy.
The results demonstrate that the MLSM is an effective and promising tool for reconstructing multiphase structures in traveltime tomography.

\section*{Acknowledgments}
Wenbin Li is supported by the Natural Science Foundation of Shenzhen (JCYJ20240813104841055), and the Fundamental Research Funds for the Central Universities (HIT.OCEF.2024017). The work of Leung was supported in part by the Hong Kong RGC grants 16302223 and 16300524.

\section*{Declaration of generative AI and AI-assisted technologies in the manuscript preparation process}

During the preparation of this work, the authors used AI to improve the clarity, grammar, and overall flow of the language. After using this tool, the author(s) reviewed and edited the content as needed and take(s) full responsibility for the content of the published article.

\bibliographystyle{plain}
\bibliography{myref}

\end{document}